\documentclass{amsart}
\usepackage[dvipdfm]{graphicx}
\usepackage{fancybox}
\usepackage[all]{xy}
\title [Veech holomorphic families of Riemann surfaces] {Veech holomorphic families of Riemann surfaces, holomorphic sections, and Diophantine problems}
\date{}

\author{Yoshihiko Shinomiya}
\address{Department of Mathematics
Tokyo Institute of Technology
2-12-1 Ookayama, Meguro-ku, Tokyo 152-8551, JAPAN}
\email{shinomiya.y.aa@m.titech.ac.jp}

\begin{document}
\maketitle
\bibliographystyle{alpha} 

\theoremstyle{plain}
\newtheorem{theorem}{Theorem}[section]

\theoremstyle{definition}
\newtheorem{definition}[theorem]{Definition}

\theoremstyle{plain}
\newtheorem{proposition}[theorem]{Proposition}

\theoremstyle{plain}
\newtheorem{lemma}[theorem]{Lemma}

\theoremstyle{plain}
\newtheorem{corollary}[theorem]{Corollary}

\theoremstyle{definition}
\newtheorem{example}[theorem]{Example}

\theoremstyle{remark}
\newtheorem*{remark}{\bf Remark}

\theoremstyle{remark}
\newtheorem*{problem}{\bf Problem}
\begin{abstract}
In this paper, we construct holomorphic families of Riemann surfaces from 
Veech groups and characterize their holomorphic sections by some points of corresponding flat surfaces.
The construction gives us concrete solutions for some Diophantine equations over function fields.
Moreover, we give upper bounds of the numbers of holomorphic sections of certain holomorphic families of Riemann surfaces.
\end{abstract}

\section{Introduction}
A holomorphic family $(M, \pi, B)$ of Riemann surfaces of type $(g,n)$ is a triple of a two-dimensional complex manifold $M$, a Riemann surface $B$ and a holomorphic map $\pi : M \rightarrow B$ such that  the fiber $X_t=\pi^{-1}(t)$ is a Riemann surface  of type $(g,n)$ for each $t\in B$ and the complex structure of $X_t$ depends holomorphically on the parameter $t$.
Holomorphic families of Riemann surfaces are first studied in algebraic geometry as the Diophantine problems for function fields on compact Riemann surfaces.

Let $M(B_0)$ be the field of all meromorphic functions on a compact Riemann surface $B_0$.
We take a three-variable irreducible homogeneous polynomial $f$ whose coefficients are in $M(B_0)$.
The Diophantine problem is to find all solutions $[X:Y:Z]\in \mathbb{P}^2(M(B_0))$ of the Diophantine equation $f(X,Y,Z)=0$.
It is known that a Diophantine equation defines holomorphic family of Riemann surfaces and
all the solutions of a Diophantine equation
correspond to all the holomorphic sections of the corresponding holomorphic family of Riemann surfaces.
By using Teichm\"uller theory, Imayoshi and Shiga \cite{ImaShi88} proved that
 the number of holomorphic sections of a locally non-trivial holomorphic family of Riemann surfaces is finite.
And Shiga \cite{Shiga97} estimated the number of holomorphic families of Riemann surfaces of certain types.
In this paper, we consider holomorphic families of Riemann surfaces obtained by Veech groups.

Let $X$ be a Riemann surface of type $(g, n)$. 
The Veech group $\Gamma(X,u)$ is the group of all elements in ${\rm PSL}(2, \mathbb{R})$ which induce affine homeomorphisms on a flat surface $(X,u)$.
A Veech group induces a Teichm\"uller disk which is a holomorphic isometric embedding of a hyperbolic plane $\mathbb{H}$ into the Teichm\"uller space of the Riemann surface $(X,u)$.
Veech \cite{Veech89} proved that the mirror Veech group $\bar \Gamma(X,u)$
 is a representation of  the action of mapping class group into the Teichm\"uller disk defined by $(X,u)$.
And the action is regarded as the action of Fuchsian group $\bar \Gamma(X,u)$ into $\mathbb{H}$.
Hence, $\mathbb{H}/\bar \Gamma(X,u)$ is holomorphically embedded into the moduli space $M(g, n)$. 
In this paper, we construct holomorphic families of Riemann surfaces from such  embeddings and observe their properties.
We call such holomorphic families of Riemann surfaces {\it Veech holomorphic families of Riemann surfaces}.
Especially, we study their holomorphic sections.

The paper is organized as follows.
In section  \ref{Preliminaries}, we define fundamental objects that we use in this paper, and observe their properties.
In section \ref{Veech}, we define Veech groups and show that orbifolds induced by Veech groups are embedded holomorphically and locally isometrically into moduli spaces. 
Moreover, a relation between geodesic flows on a flat surface and the Veech group is explained in this section.
A construction of Veech holomorphic families of Riemann surfaces is discussed in section \ref{vhf}.
Also, representations and monodromies of Veech holomorphic families of Riemann surfaces are observed in this section.
Section \ref{section} characterizes holomorphic sections of Veech holomorphic families of Riemann surfaces by some points of corresponding flat surfaces. 
Also, by applying this characterization, we find all the holomorphic sections of certain Veech holomorphic families of Riemann surfaces. 
In section \ref{Diophantine}, we construct examples of Diophantine equations from some Veech holomorphic families of Riemann surfaces.
All solutions of the Diophantine equations are given by using the characterization of holomorphic sections.
Section \ref{number} gives an upper bound of the number of holomorphic sections for certain Veech holomorphic families of Riemann surfaces.
The upper bound depends only on topological quantities of the base Riemann surface $B$ and fiber Riemann surfaces $X_t=\pi^{-1}(t)$ of holomorphic families of Riemann surfaces $(M, \pi, B)$.
To obtain the upper bound, we show a certain property of Fuchsian groups (Theorem \ref{fuchsian}), which may have its own interest.

\section{Preliminaries}\label{Preliminaries}
In this section, we see some properties of Teichm\"uller spaces, Kobayashi pseudo-metrics, Teichm\"uller modular groups, Bers fiber spaces and Fuchsian groups.

Let $X$ be a Riemann surface of type $(g, n)$ with $3g-3+n>0$.

\subsection{Teichm\"uller spaces}
Let us consider all pairs $(Y, f)$ of Riemann surfaces $Y$ of type $(g, n)$ and quasiconformal maps $f$ from $X$ onto $Y$.
Two such pairs $(Y_1,f_1)$ and $(Y_2, f_2)$ are said to be equivalent if there exists a conformal map $h: Y_1 \rightarrow Y_2$ which is homotopic to $f_2 \circ f_1^{-1}$.
This equivalence relation is called the Teichm\"uller equivalence. 
\begin{definition}[Teichm\"uller space]
The Teichm\"uller space $T(X)$ of $X$ is the set of all equivalence classes of such pairs.
We denote by $[Y, f]$ the equivalence class of a pair $(Y, f)$.
The point $[X, id: X\rightarrow X]$ is called the base point of $T(X)$.
\end{definition}
For any two points $[Y_1, f_1], [Y_2,f_2] \in T(X)$, we set
\begin{center}
$\displaystyle d_T \left([Y_1,f_1], [Y_2, f_2]\right) = \inf_h \log K(h)$.
\end{center}
Here, the infimum is taken over all quasiconformal maps $h: Y_1 \rightarrow Y_2$ which are homotopic to $f_2 \circ f_1^{-1}$ and $K(h)$ is the dilatation of $h$.
That is, let $\mu_h=\bar \partial h/ \partial h$ be the Beltrami coefficient of $h$.
By the definition of quasiconformal maps, the essential supremum norm $|| \mu_h||_\infty$ of $h$ is less than 1.
The dilatation $K(h)$ of $h$ is defined by $K(h)=\frac{1+||\mu_h||_\infty}{1-||\mu_h||_\infty}$. 
Then $d_T$ becomes a metric on $T(X)$.
\begin{definition}[Teichm\"uller metric]
The metric $d_T$ is called the Teichm\"uller metric.
\end{definition}
We consider geodesics of $T(X)$ with respect to the Teichm\"uller metric $d_T$.
\begin{definition}[Holomorphic quadratic differential]
A holomorphic quadratic differential $q$ on $X$ is a tensor whose restriction to every coordinate neighborhood $(U,z)$ is the form $fdz^2$, here $f$ is a holomorphic function on $U$.\\
We define $|q |$ to be the differential 2-form on $X$ whose restriction to every coordinate neighborhood $(U,z)$ has the form $|f|dxdy$ if $q$ equals $fdz^2$ in $U$. 
We say $q$ is integrable if its norm
\begin{eqnarray}
||q ||=\int\!\!\!\int_X|q| \nonumber 
\end{eqnarray}
is finite. 
Let $\bar X$ be the compact Riemann surface of genus $g$ obtained by filling all punctures of $X$.
We construct a function ${\rm ord} : \bar X \rightarrow \mathbb{Z}$.
If $z \in \bar X$ is a zero of $q$ of order $n$, we set ${\rm ord}(z)=n$.
If $z\in \bar X$ is a pole of $q$ of order $n$, we set ${\rm ord}(z)=-n$.
And, we set ${\rm ord}(z)=0$ if $z\in \bar X$ is not a zero or pole of $q$.
\end{definition}
\begin{remark}
Note that $q$ is integrable if and only if 
${\rm ord}(z)\geq -1$ for all $z \in \bar X$.
And, if $q \not = 0$, then by the Riemann-Roch theorem, we have 
\begin{center}
$\displaystyle \sum_{z \in \bar X} {\rm ord}(z)=4g-4$.
\end{center}
\end{remark}

\begin{definition}[Teichm\"uller map]
A quasiconformal map $f: X \rightarrow Y$  is said to be a Teichm\"uller map if its Beltrami coefficient $\mu_f$ is the form
\begin{center}
$\mu_f=k \frac{|q|}{q}$
\end{center} 
for some integrable holomorphic quadratic differential $q$ on $X$ and $k$ $ (0\leq k<1)$.  
\end{definition}
The next two theorems claim that there is an unique Teichm\"uller map from $X$ to $Y$ realizing the distance between the base point $[X, id]$ and each point $[Y, f]$.
\begin{theorem}[Teichm\"uller's uniqueness theorem]\label{Teich1}
Let $f_0: X \rightarrow  Y$ be a Teichm\"uller map and  $\mu_{f_0}=k \frac{|q|}{q}$ for some integrable holomorphic quadratic differential $q$ on $X$ and $k$ $ (0\leq k<1)$.
Then, for every quasiconformal map $f:X \rightarrow Y$ which is homotopic to $f_0$, the inequality
\begin{center}
$||\mu_f||_\infty\geq ||\mu_{f_0}||_\infty=k$
\end{center}
holds.
Where, $\mu_f$ is the Beltrami coefficient of $f$.
The equality holds if and only if $f=f_0$.
\end{theorem}
\begin{theorem}[Teichm\"uller's existence theorem]\label{Teich2}
Let $f: X\rightarrow Y$ be a quasiconformal map.
There exists a Teichm\"uller map $f_0 : X \rightarrow Y$ which is homotopic to $f$.
\end{theorem}
Therefore, we conclude that the Teichm\"uller distance between the base point $[X, id]$ and each point $[Y, f]$ 
is 
\begin{center}
$d_T([X, id], [Y, f])= \log K(f_0)=\log \frac{1+k}{1-k}$
\end{center}
if the Beltrami coefficient $\mu_{f_0}$ of the unique Teichm\"uller map $f_0: X\rightarrow Y$ which is homotopic to  $f$
is the form $k\frac{|q|}{q}$.
For each geodesic ray $r:[0, \infty) \rightarrow T(X)$ with $r(0)=[X, id]$, there exists an unique  holomorphic quadratic differential $q$ on $X$ such that $r(t)=[Y_t, f_t]$ if we set $f_t:X\rightarrow Y_t$ to be the Teichm\"uller map with the Beltrami coefficient $\mu_{f_t}=\frac{t-1}{t+1}\frac{|q|}{q}$.
Therefore, $d_T$ is a complete metric.
Since the complex dimension of the Banach space of all holomorphic quadratic differentials on $X$ is $3g-3+n$, $T(X)$ is homeomorphic to $\mathbb{C}^{3g-3+n}$ and hence, $T(X)$ is proper.
For the proof of them, see \cite{ImaTan92}.

The Teichm\"uller space $T(X)$ of $X$ has a complex structure induced by the Bers embedding which is an embedding of $T(X)$ into a certain  Banach space of holomorphic quadratic differentials on $X$.
The Teichm\"uller space is embedded into the Banach space as a bounded domain. 
See \cite{ImaTan92}.

\subsection{Kobayashi pseudo-metrics}
On a complex manifold, we can define a pseudo-metric which is call the Kobayashi pseudo-metric (\cite{Kobayashi05}).  
Let $M$ be a connected complex manifold.
Given two points $z, w \in M$, we choose points $z=z_0, z_1, \cdots, z_{k-1}, z_k=w$ of $M$, points $a_1, b_1, \cdots, a_k, b_k$ of a unit disk $\mathbb{D}$, and holomorphic maps $f_1, \cdots, f_k$ from $\mathbb{D}$ to $M$ with $f_i(a_i)=z_{i-1}$ and $f_i(b_i)=z_i$ 
for $i=1, \cdots, k$.
For each choice of points and maps as above, we consider the number
\begin{center}
$d_{\mathbb{D}}(a_1,b_1)+ \cdots + d_{\mathbb{D}}(a_k,b_k)$,
\end{center}
where $d_\mathbb{D}$ is the hyperbolic metric on $\mathbb{D}$.
Let $d_M^K(z, w)$ be the infimum of the numbers obtained as above.
It is easy to see that  $d_M^K: M \times M \rightarrow \mathbb{R}$ is continuous  and satisfies the axioms of pseudo-metric:
\begin{center}
$d_M^K(z,w) \geq 0$, $d_M^K(z,w)=d_M^K(w,z)$ and $d_M^K(z,r)+d_M^K(r, w) \geq d_M^K(z, w)$.
\end{center}

\begin{definition}[Kobayashi metric, Kobayashi hyperbolic manifold]
The pseudo-metric $d_M^K$ on $M$ is called the Kobayashi pseudo-metric of $M$.
If $d_M^K$  satisfies the axioms of metric, we call $d_M^K$ the Kobayashi metric of $M$, and $M$ a Kobayashi hyperbolic manifold.
If a complex manifold $M$ is a Kobayashi hyperbolic manifold and the Kobayashi metric $d_M^K$ is proper, we call $M$ a proper Kobayashi hyperbolic manifold.  
\end{definition}
By definition, Kobayashi pseudo-metric has a distance decreasing property.
That is, for complex manifolds $M$, $N$ and a holomorphic map $f: M \rightarrow N$, we have
\begin{center}
$d_M^K(z,w) \geq  d_N^K(f(z), f(w))$
\end{center}
for all $z,w \in M$.

Royden \cite{Royden71} prove the following theorem.
\begin{theorem}\label{Royden}
On the Teichm\"uller space $T(X)$ of $X$, the Kobayashi pseudo-metric $d_{T(X)}^K$ coincides with the Teichm\"uller metric $d_T$. 
\end{theorem}
This implies that $T(X)$ is a proper Kobayashi hyperbolic manifold.
\begin{remark}
If $M$ is a hyperbolic Riemann surface, $d_M^K$ coincides with the complete hyperbolic metric which is compatible with its complex structure.
\end{remark}
\begin{proposition}[\cite{Kobayashi05}]\label{Kobayashi}
Let $M$ be a complex manifold and $\pi: \widetilde M \rightarrow M$ be a covering map.
Choose any two points $z,w \in M$ and fix $\tilde z \in \pi^{-1}(z)$.
Then
\begin{center}
$\displaystyle d_M^K(z, w) = \inf_{\tilde w \in \pi^{-1}(w)}d_{\widetilde M}^K(\tilde z, \tilde w)$.
\end{center}
\end{proposition}
\begin{proof}
Since $\pi$ is holomorphic, we have 
\begin{center}
$\displaystyle d_M^K(z, w) \leq  \inf_{\tilde w \in \pi^{-1}(w)}d_{\widetilde M}^K(\tilde z, \tilde w)$.
\end{center}
Assume that there exists $\epsilon >0$ such that
\begin{center}
$\displaystyle d_M^K(z, w)  +\epsilon <\inf_{\tilde w \in \pi^{-1}(w)}d_{\widetilde M}^K(\tilde z, \tilde w)$.
\end{center}
By the definition of Kobayashi pseudo-metric, there exist $a_1, b_1, \cdots, a_k, b_k \in \mathbb{D}$ and holomorphic maps $f_1, \cdots, f_k$ such that
\begin{center}
$z=f(a_1)$, $f(b_i)=f(a_{i+1})$ $(i=1, \cdots, k-1)$, $f_k(a_k)=w$ 
\end{center}
and 
\begin{center}
$\displaystyle d_M^K(z, w)+ \epsilon > \sum_{i=1}^k d_{\mathbb{D}}(a_i, b_i)$.
\end{center}
We lift $f_1, \cdots, f_k$ to holomorphic maps $\tilde f_1, \cdots, \tilde f_k$ from $\mathbb{D}$ to $\widetilde M$
such that 
$\tilde z= \tilde f_1(a_1)$ and  $\tilde f(b_i)= \tilde f(a_{i+1})$ $(i=1, \cdots, k-1)$.
If we set $\tilde w=\tilde f_k(b_k)$, then $\pi(\tilde w)=w$ and 
\begin{center}
$\displaystyle d_{\widetilde M}^K(\tilde z, \tilde w) \leq \sum _{i=1}^k d_{\mathbb{D}}(a_i, b_i)$.
\end{center}
Hence, $d_{\widetilde M}^K(\tilde z, \tilde w) < d_M^K(z, w)+ \epsilon$.
This contradicts the assumption. 
\end{proof}

\subsection{Teichm\"uller modular groups}
Let ${\rm Mod}(X)$ be the group of all homotopy classes of quasiconformal self-maps of $X$.
We denote by $[h]$ the equivalence class of a quasiconformal self-map $h$ of $X$.
\begin{definition}
The group ${\rm Mod}(X)$ is called the the Teichm\"uller modular group of $X$.
\end{definition}
The Teichm\"uller modular group ${\rm Mod}(X)$ acts on $T(X)$ as follows.
For each $[h]\in {\rm Mod}(X)$, a map $[h]_*: T(X)\rightarrow T(X)$ is defined by
\begin{center}
$[h]_*[Y,f]=[Y, f \circ h^{-1}]$
\end{center}
for all $[Y, f] \in T(X)$.
Clearly, this action is isometric with respect to the Teichm\"uller metric $d_T$.
Moreover, it is known that $[h]_*$ is a holomorphic  self-map of $T(X)$ and the action of ${\rm Mod}(X)$ is properly discontinuous. 
See \cite{ImaTan92}.
Also, the quotient space $T(X)/{\rm Mod}(X)$ coincides with the moduli space $M(g, n)$ of Riemann surfaces of type $(g, n)$.

\subsection{Bers fiber spaces}\label{bfs}
Let $p: \mathbb{H} \rightarrow X$ be the universal covering map with the covering transformation group $G$.
Let $f: X \rightarrow Y$ be a quasiconformal map with Beltrami coefficient $\mu_f$.
We define the Beltrami coefficient $\tilde \mu_f$ on $\hat{\mathbb{C}}$ which is the lift of $\mu_f$ on $\mathbb{H}$ and equals $0$ on $\hat{\mathbb{C}}-\mathbb{H}$.
Let $w^{\mu_f}: \hat{\mathbb{C}} \rightarrow \hat{\mathbb{C}}$ be the unique $\tilde \mu_f$-quasiconformal map  which fixes $0, 1, \infty$.
Then two domains $w^{\mu_{f_1}}(\mathbb{H})$ and $w^{\mu_{f_2}}(\mathbb{H})$ are equal if two quasiconformal maps $f_1: X \rightarrow Y_1$ and $f_2: X \rightarrow Y_2$ are Teichm\"uller equivalent.
We denote by $\mathbb{H}^t$ the domain $w^{\mu_f}(\mathbb{H})$ for $t=[Y, f] \in T(X)$.
Then
\begin{center}
$F(G)=\{(t, z) : t \in T(X),\ z \in \mathbb{H}^t\}$
\end{center}
is called the Bers fiber space over $T(X)$.
We induce the complex structure on  $F(G)$ as a submanifold of  $T(X) \times \left(\mathbb{C}-\{0, 1\}\right)$.
Bers prove the following.
\begin{theorem}[\cite{Bers73}]\label{Bers}
Let us fix a point $a \in X$ and set $\dot{X}=X-\{a\}$.
The Bers fiber space $F(G)$ is biholomorphic to $T(\dot{X})$.
\end{theorem}

\subsection{Holomorphic families of Riemann surfaces}\label{families}
\begin{definition}[Holomorphic family of Riemann surfaces]
Let $\bar M$ be a two-dimensional complex manifold and $A$ be a  one-dimensional analytic subset of $\bar M$ or empty.
Let $B$ be a Riemann surface. 
Assume that there exists a proper holomorphic map $\bar \pi: \bar M \rightarrow B$ satisfying the following two conditions.
\begin{enumerate} 
\item[(1)] By setting $M=\bar M -A$ and $\pi=\bar \pi  |_M$, the holomorphic map $\pi$ is of maximal rank at every point of $M$. 
\item[(2)] The fiber $X_t=\pi^{-1}(t)$ over each $t \in B$ is a Riemann surface of fixed finite type $(g, n)$
. 
\end{enumerate} 
We call such triple $(M, \pi, B)$ a holomorphic family of Riemann surfaces of type $(g, n)$ over $B$.
We assume that $3g-3+n>0$.
\end{definition}
Let $(M,\pi,B)$ be a holomorphic family of Riemann surfaces of type $(g, n)$.
Let $M(g, n)$ be the moduli space of Riemann surfaces of type $(g, n)$.
Then we have a holomorphic map $\phi: B \ni t \mapsto X_t =\pi^{-1}(t)\in  M(g, n)$.
Let $p : \mathbb{H} \rightarrow B$ be the universal covering map with the covering transformation group $\Gamma<{\rm PSL}(2, \mathbb{R})$.
We fix $t_0 \in B$ and set $X=\pi^{-1}(t_0)$. 
Since $M(g, n)=T(X)/{\rm Mod}(X)$, we have a lift $\Phi : \mathbb{H}\rightarrow T(X)$ of $\phi$. 
The holomorphic map  $\Phi$ induces a homomorphism $\chi : \Gamma \rightarrow {\rm Mod}(X)$ such that $\Phi \circ A= \chi(A) \circ \Phi $ for every $A \in \Gamma$.
\begin{definition}[Representation and monodromy]
We call the holomorphic map $\Phi$ a representation of $(M, \pi, B)$ into $T(X)$ and the homomorphism $\chi$ the monodromy with respect to $\Phi$. 
\end{definition}

\begin{definition}[Locally triviality and locally non-triviality]
A holomorphic family of Riemann surfaces $(M, \pi, B)$ is called locally trivial if the induced map $\phi: B \rightarrow M(g,n)$ is constant. 
And, $(M, \pi, B)$ is called locally non-trivial if the induced map $\phi: B \rightarrow M(g,n)$ is non-constant. 
\end{definition}

\subsection{Ford region and Shimizu's lemma}
Let $\Gamma \subset  {\rm PSL}(2, \mathbb{R})$ be a Fuchsian group of type $(p,k : \nu_1, \cdots, \nu_k)$ $(\nu_i \in \{2, 3, \cdots, \infty\})$.
That is, $\mathbb{H}/\Gamma$ has genus $p$ and cone points whose orders are $\nu_1, \cdots, \nu_k$.
Now, a cone point of order $\infty$ means a puncture of $\mathbb{H}/\Gamma$.
\begin{definition}[Fundamental domains]
A domain $D$ in $\mathbb{H}$ is called a fundamental domain for $\Gamma$ if it satisfies the followings:
\begin{enumerate} 
\item[(1)] For all $A \in \Gamma$ $(A \not = id)$, $A(D)\cap D = \phi $, and 
\item[(2)] For all $z \in \mathbb{H}$, there exists $A \in \Gamma$ such that $A(z) \in {\rm cl}(D)$.
\end{enumerate} 
Here, ${\rm cl}(D)$ is the closure of $D$.
\end{definition}

Suppose that $\Gamma$ contains $ \tiny \left[\left(
  \begin{array}{cccc}
    1 & 1 \\
    0 & 1
  \end{array}
  \right)
  \right]$ and it is not a power of other elements of $\Gamma$.
We set $\Gamma_0=\{{\tiny \left[\left(
  \begin{array}{cccc}
    a & b \\
    c & d
  \end{array}
  \right)
  \right]} \in \Gamma: c \not =0\}$, and 
for all   
  $ A\in \Gamma_0$, 
\begin{center}
$D_A=\{ z \in\mathbb{H}: |z+\frac{d}{c}|>\frac{1}{|c|}\}$.
\end{center}
\begin{definition}[Ford regions]
The region
\begin{center}
$\displaystyle D=\{ z \in\mathbb{H}: |{\rm Re}(z)|<\frac{1}{2}\} \cap \bigcap_{A \in \Gamma_0}D_A$
\end{center}
is called the Ford region of $\Gamma$.
\end{definition}
\begin{proposition}\label{fundamental}
The Ford region $D$ is a fundamental domain for $\Gamma$.
\end{proposition}
For a proof of proposition \ref{fundamental}, see \cite{Ford25}.
\begin{remark}
It is clear that the hyperbolic area of $\mathbb{H}/\Gamma$ induced by the universal covering map $\mathbb{H} \rightarrow \mathbb{H}/\Gamma$ equals to that of the Ford region $D$ in $\mathbb{H}$.
And it is known that the Ford region $D$ becomes a finite sided geodesic polygon in $\mathbb{H}$.
\end{remark}
\begin{lemma}[Shimizu's lemma]\label{Shimizu}
For all ${\tiny \left[\left(
  \begin{array}{cccc}
    a & b \\
    c & d
  \end{array}
  \right)
  \right]} \in \Gamma_0$, the inequality $|c|\geq1$ holds.
\end{lemma}
For a proof, see \cite{ImaTan92}.
\section{Flat surface and Veech groups}\label{Veech}
In this section, we define Veech groups and see  some applications of Veech groups.
Let $X$ be a Riemann surface of type $(g, n)$ and $C$ be a set of all punctures of $X$ and finitely many points of $X$. 
\begin{definition}[Flat structure and flat surface]
A flat structure ${\it u}$ on $X$ with critical points on $C$ is an atlas of $X^{\prime}=X-C$ which satisfies the following conditions.
\begin{enumerate} 
\item[(1)] Local coordinates of $u$ are compatible with the orientation on $X$ induced by its complex structure.
\item[(2)] For coordinate neighborhoods $(U,z)$ and $(V,w)$ of $ u$ with $U\cap V \not=\phi $, the transition function is the form 
\begin{center}
$w=\pm z+c$ 
\end{center}
in $z(U\cap V)$ for some $c \in \mathbb{C}$.
\item[(3)]  The atlas ${\it u}$ is maximal with respect to $(1)$ and $(2)$.
\end{enumerate}
A pair $(X, u)$ of a Riemann surface $X$ and a flat structure $u$ with critical points on $C$  is called a flat surface with critical points on $C$.
Each point in $C$ is called a critical point of $(X,u)$.
In this paper, we assume that the area ${\rm Area}(X,u)$ of $X$ with respect to $u$ is finite.
\end{definition}

We construct a flat structure $u_q$ on $X$ which is compatible with the complex structure of $X$ from a holomorphic quadratic differential $q$ on $X$.
We set $C$ to be the set of all punctures of $X$ and all zeros of $q$.
For each $p_0\in X^{\prime}=X-C$, we can choose an open neighborhood $U$ such that
\begin{eqnarray}
z(p)=\int_{p_0}^{p}\sqrt q \nonumber
\end{eqnarray} 
is well-defined and  an injective function in $U$. This function is holomorphic in $U$ since $q$ is a holomorphic quadratic differential on $X$. 
If $(U,z)$ and $(V,w)$ are pairs of such neighborhoods and functions with $ U \cap V\not=\phi$, then we have $dw^2=q=dz^2$ in $U \cap V$. 
Hence, $w=\pm z+c$ in $z(U\cap V)$ for some $c \in \mathbb{C}$. 
The flat structure $u_q$ is the maximal flat structure which contains such pairs.
The set $C$ is the critical point set of $u_q$.
And the area ${\rm Area}(X, u_q)$ equals to  the norm $||q ||$ of $q$.
Conversely, if $u=\{(U_\lambda, z_\lambda)\}$ is a flat structure on $X$ which is compatible with the complex structure of $X$ and the critical point set of $u$ is $C$, then $q=\{dz_\lambda^2\}$ is a meromorphic quadratic differential on $X$ with $||q||={\rm Area}(X, u)$.
All zeros and poles of $q$ are contained in $C$.
If $p \in C$ is a zero of $q$ with ${\rm odr}(p)=n$, then the angle around $p$ with respect to $u$ is $(n+2)\pi$. 
And the angle around a single pole of $q$ with respect to $u$ is $\pi$.

Fix an integrable holomorphic quadratic differential $q$ on $X$.
Let $u=u_q=\{(U_\lambda, z_\lambda)\}$ be the flat structure induced by $q$ and $C$ be the critical point set of $u$. 
For each $A \in {\rm SL}(2, \mathbb{R})$,  we denote by $T_A$ the linear map on $\mathbb{C}=\mathbb{R}^2$  corresponding to $A$.
Then $u_A=\{(U_\lambda, T_A\circ z_\lambda)\}$ is also a flat structure on $X$ with critical points on $C$.
Note that $(X, u_A)$ is not conformal equivalent to the original Riemann surface $X$ in general. 
We obtain a map $\hat \iota : {\rm SL}(2, \mathbb{R})\rightarrow T(X)$ such that $\hat \iota(A)=[(X, u_A), id: X \rightarrow (X, u_A)]$.
Here, we consider $u_A$ as a new complex structure on $X$.
By the following lemma, $\hat \iota$ induces a map $\iota : {\rm SO}(2)\setminus {\rm SL}(2, \mathbb{R}) \rightarrow T(X)$.
\begin{lemma}\label{map1}
$\hat \iota(A)=\hat \iota (B)$ if $BA^{-1}\in {\rm SO}(2)$.
\end{lemma}
\begin{proof}
Let $A,B \in {\rm SL}(2,\mathbb{R})$ with $BA^{-1}\in {\rm SO}(2)$.
For each $(U, z) \in u$, $(T_B \circ z) \circ (T_A \circ z)^{-1}=T_B \circ T_{A^{-1}}=T_{BA^{-1}}$ is a conformal map. 
Hence, $id : (X, u_A) \rightarrow (X, u_B)$ is a conformal map.
\end{proof}
The identification of ${\rm SO}(2)\setminus {\rm SL}(2, \mathbb{R})$ with $\mathbb{H}$ by the bijective map 
\begin{center}
${\rm SO}(2)\setminus {\rm SL}(2, \mathbb{R}) \ni {\rm SO}(2)\cdot A \mapsto -\overline{A^{-1}(i)} \in \mathbb{H}$
\end{center}
 induces a map $\iota : \mathbb{H} \rightarrow T(X)$. 
Here, $A^{-1}(\cdot )$ is a M\"obius transformation.
It is known that this map $\iota$ is a holomorphic isometric embedding from a hyperbolic plane $\mathbb{H}$ into the Teichm\"uller space $T(X)$ with the Teichm\"uller distance. 
(See \cite{HerSch07}.) 
And $\iota$ satisfies $\iota(i)=[X, id]$.
Moreover, it is easy to verify that  the Beltrami coefficient of the Teichm\"uller map from $\iota(i)$ to $\iota(t)$ is $\frac{i-t}{i+t}\frac{|q|}{q}$. 
Conversely, we have the following.
\begin{proposition}\label{embtohqd}
Every holomorphic isometric embedding $\iota: \mathbb{H} \rightarrow T(X)$ with $\iota(i)=[X,id]$ is constructed from some flat structure as above.
\end{proposition}
\begin{proof}
Let $\iota: \mathbb{H} \rightarrow T(X)$ be a holomorphic isometric embedding.
Then the geodesic ray  $r(t)=it$ $(t \geq 1)$ in $\mathbb{H}$ maps to a geodesic ray in $T(X)$.

By theorem \ref{Teich1} and theorem \ref{Teich2}, there exists a holomorphic quadratic differential $q$ on $X$ such that the Teichm\"uller map form $[X, id ]$ to $\iota \circ r(t)$ has the Beltrami coefficient  $\frac{1-t}{1+t} \frac{|q|}{q}$ for all $t \geq 1$.
Note that these Teichm\"uller maps are deformations by 
$
\tiny{\left(
  \begin{array}{cccc}
   \sqrt{t} & 0  \\
    0 & 1/\sqrt{t}
  \end{array}
  \right)}$ of the flat structure $u_q$ defined by $q$.
Let   $\iota_q: \mathbb{H} \rightarrow T(X)$ be the holomorphic isometric embedding constructed from $u_q$.
Since $\iota(it)=\iota_q(it)$ for all $t \geq 1$, we conclude that $\iota=\iota_q$.
\end{proof}
\begin{definition}[Teichm\"uller disk]
We call the image of a holomorphic isometric embedding of $\mathbb{H}$ into $T(g,n)$ a Teichm\"uller disk.
\end{definition}
Let $(X, u)$ be a flat surface with critical points on $C$.
Let $\iota :\mathbb{H} \rightarrow T(X)$ be a  holomorphic isometric embedding constructed from $(X, u)$. 
We consider the image of  the Teichm\"uller disk $\Delta=\iota(\mathbb{H})$ into the moduli space $M(g, n)$ and  hence, the subgroup ${\rm Stab}(\Delta)$ of ${\rm Mod}(X)$ which consists of all elements of ${\rm Mod}(X)$ which fixing $\Delta$.
For this, we define the affine group ${\rm Aff}^+(X, u)$ of $(X,u)$.
\begin{definition}[Affine group of $(X,u)$]
The affine group ${\rm Aff}^+(X, u)$ of the flat surface $(X,u)$ is the group of all quasiconformal maps $h$ of $X$ onto itself which satisfy $h(C)=C$ and are affine with respect to the flat structure $u$. 
This means that for $(U,z), (V,w) \in u$ with $h(U) \subseteq V$, the homeomorphism $w\circ h\circ z^{-1}$ is the form $z \mapsto Az+c$ for some $A \in {\rm GL} (2,\mathbb{R})$ and $c \in \mathbb{C}$. 
Each element in ${\rm Aff}^+(X, u)$ is called an affine homeomorphism on $(X, u)$.
\end{definition}

For each $h \in {\rm Aff}^+(X, u)$, the derivative $A \in {\rm GL} (2, \mathbb{R})$ of  the affine map $w\circ h\circ z^{-1}$ is uniquely determined up to the sign since $u$ is a flat structure.
And $A$ is always in ${\rm SL}(2,\mathbb{R})$ since ${\rm Area}(X, u) =\int_X |q|=\int_X h^*(|q|)=\det(A)\cdot{\rm Area}(X, u)$. 
Here, $q$ is a holomorphic quadratic differential corresponding to $(X, u)$.
Thus, we have a group homomorphism

\begin{center}
$D: {\rm Aff}^+(X,u) \rightarrow {\rm PSL}(2,\mathbb{R})$.
\end{center}
We call this homomorphism the derivative map.

\begin{definition}[Veech group of $(X, u)$]
We call $\Gamma(X,u)=D({\rm Aff}^+(X,u))$ the Veech group of $(X, u)$. 
\end{definition}

Veech\cite{Veech89} prove the following.
\begin{proposition}
All homomorphisms of ${\rm Aff}^+(X, u)$ are not homotopic to each other.
Therefore, ${\rm Aff}^+(X, u)$ is regarded as a subgroup of ${\rm Mod}(X)$. 
\end{proposition}

By setting $\iota(t)=[X_t, id]$ for each $t \in \mathbb{H}$, we have the following theorem.
\begin{theorem}[\cite{Veech89}, \cite{EarGar97}, \cite{HerSch07}]
The group ${\rm Stab}(\Delta)$
coincides with ${\rm Aff}^+(X, u)$.
For any $h \in {\rm Aff}^+(X, u)$ and $[X_t, id] \in \Delta$,
\begin{center}
$h_*[X_t, id] =[X_t, h^{-1}]=[X_{\bar A(t)}, id]$.
\end{center}
Here, 
$A=D(h)$, 
$
R=\tiny{\left(
  \begin{array}{cccc}
   -1 & 0  \\
    0 & 1
  \end{array}
  \right)}$, and
$\bar A=RAR^{-1}$ acts on $\mathbb{H}$ as a
 M\"obius transformation. 
\end{theorem}
\begin{corollary}[\cite{Veech89}, \cite{EarGar97}, \cite{HerSch07}]
The Veech groups $\Gamma(X, u)$ is a Fuchsian group and 
we have the following commutative diagram.
Here, $\bar \Gamma(X, u)=R\Gamma(X, u)R^{-1}$.
\[
\begin{xy}
(0,0)="a"*{\Delta
},
"a"+/r11em/="b"*{T(X)},
"a"+/d5em/="c"*{\Delta/{\rm Stab}(\Delta)},
"c"+/r11em/="d"*{M(g, n)},
"a"+/l8em/="e"*{\mathbb{H}},
"c"+/l8em/="f"*{\mathbb{H}/\bar \Gamma (X, u)},
\ar@{^{(}-_{>}} "a"+<1.2em,-.em>;"b"+<-2.1em,.em>
\ar "a"+<.em,-.8em>;"c"+<.em,.8em>^{/{\rm Stab}(\Delta)}
\ar "b"+<.em,-1.em>;"d"+<-.em,.5em>^{/{\rm Mod}(X)}
\ar@{^{(}-_{>}} "c"+<3.em,.em>;"d"+<-2.1em,.em>
\ar "e"+<1.em,.em>;"a"+<-1.em,.em>^{\iota}
\ar "e"+<.em,-.8em>;"f"+<.em,.8em>^{/\bar \Gamma (X, u)}
\ar "f"+<2.7em,.em>;"c"+<-3.0em,.em>^{\sim }
\end{xy}
\]
\end{corollary}

We construct holomorphic families of Riemann surfaces from such holomorphic embeddings $\mathbb{H}/\bar \Gamma (X, u) \rightarrow M(g, n)$.

Next, we see an application of Veech groups for a dynamical behavior of the geodesic flows on flat surfaces $(X, u)$. 
We assume that the conformal structure defined by $u$ is compatible with that of $X$. 
Let $C$ be the critical point set of $(X, u)$.
Since the transition functions of $(X, u)$ are the from $z \mapsto \pm z+c$,  geodesics on $(X, u)$ and directions of  geodesics of $(X, u)$ are well-defined. 
 That is, a geodesic $l : \mathbb{R} \rightarrow X$ on $(X, u)$ is a curve in $X^\prime=X-C$ such that 
 for all $t \in \mathbb{R}$, sufficiently small $\epsilon>0$, and $(U, z) \in u$ with $l((t-\epsilon , t+\epsilon )) \subset  U$, $z \circ l |_{(t-\epsilon , t+\epsilon )}$ is an Euclidean line segment.
 The direction $\theta(l)\in [0, \pi)$ of a geodesic $l$ is the direction of $z \circ l |_{(t-\epsilon , t+\epsilon )}$ on $\mathbb{C}$.
 A saddle connection is a path on $X$ which connects two points in $C$ and is locally an Euclidean line segment with respect to $u$.
  \begin{definition}[Jenkins-Strebel direction] 
  A number $\theta \in [0, \pi)$ is called a Jenkins-Strebel direction of $(X, u)$ if all points in $X$ lie on $\theta$-direction closed geodesics or  $\theta$-direction  saddle connections of $(X, u)$.
\end{definition}
If $\theta$ is a Jenkins-Strebel direction of $(X, u)$, the flat surface $(X, u)$ is decomposed into finitely many Euclidean cylinders by cutting along  the $\theta$-direction  saddle connections.
\begin{definition}[Simple Jenkins-Strebel direction] 
If Jenkins-Strebel direction $\theta$ decomposes $(X, u)$ into only one cylinder, we call the direction $\theta$ a simple Jenkins-Strebel direction of $(X, u)$.
\end{definition}
\begin{definition}
For an Euclidean cylinder $R$ whose length of circumference is $b$ and height is $a$,
the ratio $\frac{b}{a}$ is called the modulus of $R$ and denote it by ${\rm mod}(R)$.
\end{definition}
Veech\cite{Veech89} proved the following theorem which is called the Veech dichotomy theorem.
For the definition of the ergodicity, see \cite{KatHas96}.
\begin{theorem}
\label{dichotomy}
Let $(X, u)$ be a flat surface.
Suppose that the Veech group  $\Gamma(X,u)$ is a lattice in ${\rm PSL}(2, \mathbb{R})$, that is, $\mathbb{H}/\Gamma(X,u)$ has finite hyperbolic area.
Then for each direction $\theta$, one of the following two possibilities occurs:
\begin{itemize}
\item The direction $\theta$ is a Jenkins-Strebel direction.
Moreover, if $\theta$ decomposes $X$ into $k$ cylinders $R_1, \cdots, R_k$, the ratio ${\rm mod}(R_i)/{\rm mod}(R_j)$ is a rational number for each $i,j \in \{1, \cdots, k\}$. 
\item Every $\theta$-direction geodesic is dense in $X$.
And the $\theta$-direction geodesic flow is uniquely ergodic. 
\end{itemize}
\end{theorem}


Finally, we see two examples of Veech groups.
\begin{example}[\cite{Shinomiya12}]\label{example1}
Let $X$ be a surface constructed as figure \ref{Veechexam1-1}.
We induce an unique conformal structure on $X$ such that the quadratic differential $dz^2$ on the interior of the rectangle of figure \ref{Veechexam1-1} extends to a holomorphic quadratic differential $q$ on $X$. 
Then $X$ is a Riemann surface of type $(2,0)$ and vertices of four squares become two points in $X$. 
These points are zeros of $q$ of order 2.
Let $u=u_q$ be the flat structure defined by $q$. 
Then $\theta=0$ and $\frac{\pi}{2}$ are Jenkins-Strebel directions of $(X,u)$.
The direction $\theta=\frac{\pi}{2}$ decomposes $(X, u)$ into two cylinders.
And the direction $\theta=0$ decomposes $(X, u)$ into only one cylinder.
Therefore, $\theta=0$ is a simple Jenkins-Strebel direction.
The two Jenkis-Strebel directions $0, \frac{\pi}{2}$ of $(X,u)$ imply that  
{  \tiny $\left(
  \begin{array}{cc}
    1 & 1 \\
    0 & 1 
  \end{array}
  \right)$}
 and 
 {  \tiny $\left(
  \begin{array}{cccc}
    1 & 0 \\
    2 & 1
  \end{array}
  \right)$}
define elements in ${\rm Aff}^+(X, u)$ as figure \ref{Veechexam1-2}.
Hence, 
$\Gamma=\left< {\tiny  \left[\left(
  \begin{array}{cccc}
    1 & 1 \\
    0 & 1 
  \end{array}
  \right)
  \right]}, 
{ \tiny  \left[\left(
  \begin{array}{cccc}
    1 & 0 \\
    2 & 1
  \end{array}
  \right)
  \right] }
\right>$ is a subgroup of $\Gamma(X, u)$. 

Since every element in ${\rm Aff}^+(X, u)$ must preserve the set of all critical points, $\Gamma(X, u)$ is a subgroup of ${\rm PSL}(2,\mathbb{Z})$. 
It is known that $ \tiny\left<
{  \left[ \left(\begin{array}{cccc}
    1 & 2 \\
    0 & 1 
  \end{array}
 \right) \right] }, 
 { 
  \left[\left(
  \begin{array}{cccc}
    1 & 0 \\
    2 & 1
  \end{array} \right)
  \right]}
\right>$ is the congruence subgroup of level $2$ and has index $6$ in ${\rm PSL}(2,\mathbb{Z})$. 
Hence, $\Gamma(X, u)$ is either $\Gamma$ or ${\rm PSL}(2,\mathbb{Z})$. 
However, 
  ${ \left[ \tiny 
\left(
  \begin{array}{cccc}
    1 & 0 \\
    1 & 1 
  \end{array}
  \right) \right]}$ 
cannot be an element in  $\Gamma(X, u)$. Therefore $\Gamma(X, u)$ must be $\Gamma$.
It is easy to see that $\mathbb{H}/\bar \Gamma(X, u)$ is an orbifold of genus $0$ has $2$ punctures and $1$ cone point of order $2$.
Thus, $\Gamma(X, u)$ is a lattice in ${\rm PSL}(2, \mathbb{R})$ and the flat surface $(X, u)$ satisfies the Veech dichotomy.

\begin{figure}[h]
 \begin{center}
  \includegraphics[keepaspectratio, scale=1]{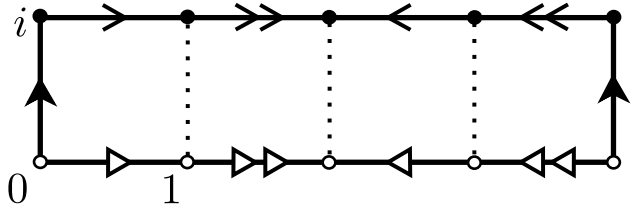}
\caption{The flat surface $(X,u)$}
\label{Veechexam1-1}
\end{center}
\end{figure}

 \begin{figure}[h]
 \begin{center}
  \includegraphics*[keepaspectratio, scale=1]{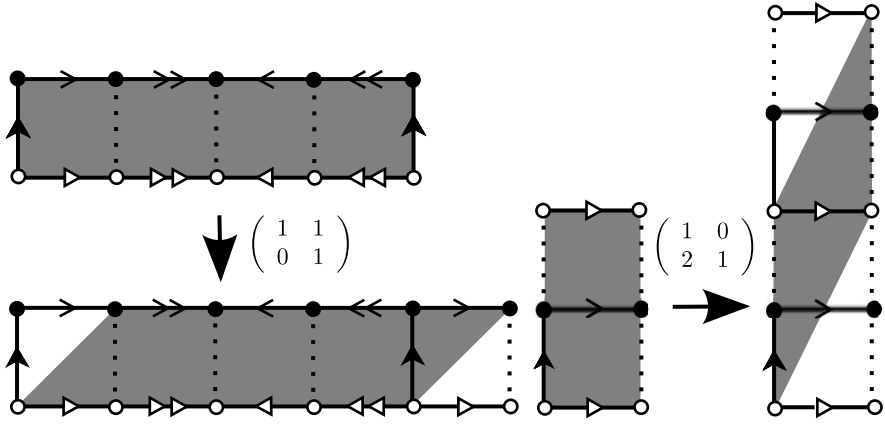}
\caption{Actions of two matrices onto $X$}
\label{Veechexam1-2}
\end{center}
\end{figure}
\end{example}

\begin{example}[\cite{EarGar97}]\label{example2}
Fix $n\geq 2$ and let $\Pi_n$ be a regular $4n$-gon.
We assume that $\Pi_n$ has two horizontal sides, lengths of the sides are 1. 
We identify each side of  $\Pi_n$ with the opposite parallel side by an Euclidean translation (see figure \ref{Veechexam2-1}) and denote the resulting surface by $P_n$.
We induce an unique conformal structure on $P_n$ such that the quadratic differential $dz^2$ on the interior of $\Pi_n$ extends to a holomorphic quadratic differential $q_n$ on $P_n$. 
Then $P_n$ is a Riemann surface of type $(\frac{n}{2}, 0)$. 
And, the vertices of $\Pi_n$ become an unique zero of $q_n$ of order $4n-4$.
Let $u_{q_n}$ be the flat structure defined by $q_n$. 
Now,
$R_n={ \tiny 
\left(
  \begin{array}{cccc}
    \cos \pi/2n  & -\sin \pi/2n \\
    \sin \pi/2n & \cos \pi/2n 
\end{array}
  \right)
}$
and 
$T_n={ \tiny
\left(
  \begin{array}{cccc}
    1 & 2\cot \pi/4n \\
    0 & 1
\end{array}
  \right)}$ 
induce elements in ${\rm Aff}^+(P_n, u_{q_n} )$. 
The action of $R_n$ on $P_n$ is the rotation about the center of $\Pi_n$ of angle $\frac{\pi}{2n}$. 
To see the action of $T_n$ on $P_n$, we cut $P_n$ along all horizontal saddle connections. 
Then $P_n$ is decomposed into $\frac{n}{2}$ cylinders and 
the action of $T_n$ is the composition of the square of the right Dehn twist along a core curve of the cylinder which contains the center of $\Pi_n$ and the right Dehn twists along core curves of the other cylinders. 
Thus, $\Gamma=\left<[R_n],[T_n]\right>$ is a subgroup of the Veech group $\Gamma(P_n, u_{q_n})$. 
It is easy to see that $\Gamma$ is a triangle group of type $(2n, \infty, \infty)$.
Since only discrete group that contains $\Gamma$ is a triangle group of type $(2, 4n, \infty)$  (see \cite{EarGar97} and \cite{Singerman72})
and this cannot be $\Gamma(P_n,u_{q_n})$, we have $\Gamma(P_n, u_{q_n})=\left<[R_n],[T_n]\right>$. 
And, $\Gamma(P_n,u_{q_n})$ is a lattice in ${\rm PSL}(2, \mathbb{R})$ and the flat surface $(P_n,u_{q_n})$ satisfies the Veech dichotomy.
\begin{figure}[h]
 \begin{center}
 \includegraphics*[keepaspectratio, scale=1]{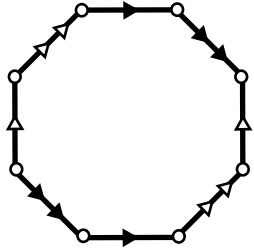}
\caption{The flat surface $(P_8, u_{q_8})$}
\label{Veechexam2-1}
 \end{center}
 \end{figure}
\end{example}

\section{Construction of Veech holomorphic families of Riemann surfaces}\label{vhf}
In this section, we construct  holomorphic families of Riemann surfaces from flat surfaces and the Veech groups. 
We call those holomorphic families {\it Veech holomorphic families of Riemann surfaces}.
And we observe representations and monodromies  of Veech holomorphic families of Riemann surfaces.

We assume that $X$ is  a Riemann surface of type $(g, n)$ with $3g-3+n>0$.
Let $(X,u)$ be a flat surface with critical points on $C$ and assume that this flat structure is compatible with the complex structure of $X$.
Let $q$ be the corresponding holomorphic quadratic differential on $X$.
In section \ref{Veech}, we constructed a holomorphic isometric embedding $\iota: \mathbb{H} \ni t \mapsto [X_t,id] \in T(X)$ from $(X, u)$.
The Beltrami coefficient of the Teichm\"uller map from $\iota(i)=[X, id]$ to $\iota(t)$ is $\mu_t=\frac{i-t}{i+t}\frac{|q|}{q}$. 
Let $\rho: \mathbb{H} \rightarrow X$ be the universal covering map with the covering transformation group  $G < {\rm PSL}(2,\mathbb{R})$. 
For each $t \in \mathbb{H}$, we denote by $\tilde \mu_t$ 
 the Beltrami coefficient on $\hat {\mathbb{C}}$ which is the lift of $\mu_t$ in $\mathbb{H}$ and equals to  0 in $\hat {\mathbb{C}}-\mathbb{H}$.
Let $w^{\tilde \mu_t} : \hat{\mathbb{C}} \rightarrow \hat{\mathbb{C}}$ be an unique $\tilde \mu_t$-quasiconformal map which fixes $0, 1, \infty$. 
We set 
\begin{center}
$\widetilde{M}=\{(t,z): t \in \mathbb{H},\ z \in w^{\tilde \mu_t}(\mathbb{H})\}$.
\end{center}
By identifying $\mathbb{H}$ with $\iota(\mathbb{H})$, $\widetilde{M}$ is a two-dimensional complex manifold as a submanifold of the Bers fiber space $F(G)$.
The covering transformation group  $G$ acts on $\widetilde M$ holomorphically as $\sigma_*(t,z)=(t, w^{\tilde \mu_t}\sigma (w^{\tilde \mu_t})^{-1}(z))$ for all $\sigma \in G$ and $(t, z) \in \widetilde M$.
Set $\widehat{M}=\widetilde M/ G$.
Then we can write $\widehat{M}$ as
\begin{center}
$\widehat{M}=\{(t,z): t \in \mathbb{H},\ z \in X_t=f_t(X)\}$.
\end{center}
Here, $f_t : X \rightarrow X_t=f_t(X)$ is the induced quasiconformal map of $w^{\tilde \mu_t}|_{\mathbb{H}}$.
Of course, $X_t$ coincides with $\iota(t)$ and $f_t$ is a Teichm\"uller map with the Beltrami coefficient $\mu_t$. 

Recall that ${\rm Aff}^+(X,u)$ is the group of all affine homeomorphisms from $(X,u)$ onto itself which fixes $C$ and the homomorphism $D: {\rm Aff}^+(X,u) \rightarrow {\rm PSL}(2,\mathbb{R})$ is the derivative map.
The affine group ${\rm Aff}^+(X,u)$ acts on $\widehat{M}$ holomorphically as
\begin{center}
$h_*(t,z)= (\bar A (t), f_{\bar A(t)}h(f_t)^{-1}(z))$
\end{center}
for all $h \in {\rm Aff}^+(X,u)$ and $(t,z) \in \widehat{M}$.
Here, $A=D(h)$, $
R=\tiny{\left(
  \begin{array}{cccc}
   -1 & 0  \\
    0 & 1
  \end{array}
  \right)}$ and
$\bar A=RAR^{-1}$ acts on $\mathbb{H}$ as a M\"obius transformation.
We set $N=\widehat{M}/{\rm Aff}^+(X,u)$ and $\bar B=\mathbb{H} /\bar \Gamma (X,u)$.
Let $\pi^\prime :N \rightarrow \bar B$ be the projection, $\mathbb{H}^*$ be $\mathbb{H}$ with all elliptic fixed points of $\bar \Gamma(X,u)$ are removed, $B= \mathbb{H}^*/\bar \Gamma(X,u)$, and $N^*=(\pi^\prime)^{-1}(B)$. 
It is easy to see that $N$ is homeomorphic to the product space $B \times \left(X/{\rm Ker}(D)\right)$.
There exists a branched covering map $M \rightarrow N$ which corresponds to the branched covering map
$B \times X \rightarrow B \times \left(X/{\rm Ker}(D)\right)$.
Now, we obtain the holomorphic family $(M, \pi, B)$ of Riemann surfaces of type $(g, n)$ over $B$.
Here, $\pi: M \rightarrow B$ is the projection.  
\begin{definition}[Veech holomorphic families]
We call the holomorphic families of Riemann surfaces $(M, \pi, B)$ which are constructed as above Veech holomorphic families of Riemann surfaces of type $(g, n)$ over $B$.
\end{definition}
\begin{remark}
The triple $(N^*, \pi^\prime, B)$ is also a holomorphic family of Riemann surfaces. 
And, $(M, \pi, B)$ and $(N^*, \pi^\prime, B)$ are locally non-trivial holomorphic families of Riemann surfaces.
\end{remark}
We observe a representation of a Veech holomorphic family $(M, \pi, B)$ of Riemann surfaces and the monodromy.
First, we have the holomorphic isometric embedding $\iota : \mathbb{H} \rightarrow T(g, n)$.
Let $\bar \rho: \mathbb{H} \rightarrow \bar B$ and $\rho : \mathbb{H}\rightarrow B$ be the universal covering maps.
Since $\mathbb{H}^*$ is a covering space of $B$, there exists a covering map $\rho_0: \mathbb{H} \rightarrow \mathbb{H}^*$ which satisfies $\rho=\bar \rho \circ  \rho_0$.
Then, the map $\Phi = \iota \circ \rho_0$ is a representation of $(M, \pi, B)$.
Since $B=\mathbb{H}^*/\bar \Gamma(X, u)$, we can identify $\bar \Gamma(X, u )$ with $\pi_1(B)/\pi_1(\mathbb{H}^*)$.
And there exists a natural map $\bar \Gamma(X, u) \rightarrow \Gamma(X, u)$.
The monodromy $\chi : \pi_1(B) \rightarrow {\rm Aff}^+(X, u)$ of the representation $\Phi$ is a lift of this natural map by the maps $\pi_1(B) \rightarrow \pi_1(B)/\pi_1(\mathbb{H}^*)$ and $D: {\rm Aff}^+(X, u) \rightarrow \Gamma(X, u)$.  
\[
\begin{xy}
(0,0)="a"*{\mathbb{H}},
"a"+/d4em/+/r3.em/="b"*{\mathbb{H}^*},
"a"+/d8em/="c"*{B},
"b"+/r6em /="d"*{T(X)},

\ar "a"+<.5em,-.6em>;"b"+<-.2em,.6em>^{\rho_0}
\ar "b"+<-.5em,-.6em>;"c"+<.3em,.6em>^{\bar \rho |_{\mathbb{H}^*}}
\ar "a"+<-.em,-.6em>;"c"+<-.em,.6em>_{\rho}
\ar "b"+<.8em,.em>;"d"+<-1.5em,.em>^{\iota}
\ar "a"+<.6em,.em>;"d"+<-.8em,.6em>^{\Phi}

\end{xy}
\]

\[
\begin{xy}
(0,-5)="a"*{\pi_1(B)},
"a"+/r11em/="b"*{{\rm Aff}^+(X, u)},
"a"+/d5em/="c"*{\pi_1(B)/\pi_1(\mathbb{H}^*)},
"c"+/r11em/="d"*{\Gamma(X, u)},
"c"+/l5.5em/="e"*{\bar \Gamma (X, u)=},
\ar "a"+<1.7em,-.em>;"b"+<-3.em,.em>^{\chi}
\ar "a"+<.em,-.8em>;"c"+<.em,.8em>
\ar "b"+<.em,-.8em>;"d"+<-.em,.8em>^D
\ar "c"+<3.5em,.em>;"d"+<-2.em,.em>
\end{xy}
\]

 \section{Holomorphic sections of Veech holomorphic families of Riemann surfaces.}\label{section}
 In this section, we characterize holomorphic sections of holomorphic families of Riemann surfaces by certain points of a corresponding flat surface.
 We use the same notations as in section \ref{vhf}.
\begin{definition}[Holomorphic sections]
Let $(M, \pi, B)$ be a holomorphic family of Riemann surfaces.
A holomorphic map $s: B \rightarrow M$ is called a holomorphic section of $(M, \pi, B)$ or a holomorphic section of $\pi$ if it satisfies $\pi \circ s =id$.
\end{definition}
Let $(M, \pi, B)$ be a Veech holomorphic family of Riemann surfaces  defined by a holomorphic quadratic differential $q$ or the corresponding flat surface $(X,u)$.
In section \ref{vhf}, we define the action of ${\rm Aff}^+(X,u)$ onto $\widehat{M}$.
We set $\widehat{N}=\widehat{M}/{\rm Ker}(D)$ and let $\hat{\pi}: \widehat{N} \rightarrow \mathbb{H}$ be the projection.
We denote each point of $\widehat{N}$ by $(t, w)$. 
Here, $t \in \mathbb{H}$ is the image of the point by the map $\hat{\pi}$ into $\mathbb{H}$  and $w$ is a point of the Riemann surface $\hat{\pi}^{-1}(t)=Y_t$.
 That is, we denote $\widehat{N}$ by
 \begin{center}
 $\widehat{N}=\{(t, w): t \in \mathbb{H},\ w \in Y_t\}$.
\end{center}
 Note that if we set $q^\prime$ to be a holomorphic quadratic differential on $Y=X/{\rm Ker}(D)$ induced by $q$, then there exists  the Teichm\"uller map $g_t: Y \rightarrow Y_t$ whose Beltrami coefficient is $\frac{i-t}{i+t}\frac{|q^\prime|}{q^\prime}$ for each $t \in \mathbb{H}$.
 
Let $s: B \rightarrow M$ be a holomorphic section of $(M, \pi, B)$.
We project this holomorphic section to a holomorphic section $s^\prime: B \rightarrow N^*$ of $(N^*, \pi^\prime, B)$.
Where, $(N^*, \pi^\prime, B)$ is a holomorphic family of Riemann surfaces constructed in section \ref{vhf}.
We lift $s^\prime$ to  a holomorphic section $\hat{s}: \mathbb{H}^* \rightarrow \widehat{N}^*=\hat{\pi}^{-1}(\mathbb{H}^*) \subset \widehat{N}$ of $\hat{\pi}|_{\widehat{N}^*}$.

\[
\begin{xy}
(0,0)="a"*{\widehat{N} \supset \widehat{N}^*},
"a"+/r6em/="b"*{N^*},
"a"+/d3.5em/="c"*{\mathbb{H}\supset \mathbb{H}^*},
"b"+/d3.5em/="d"*{B},
"b"+/u3.5em/="e"*{M},
\ar "a"+<1.em,-.7em>;"c"+<1.em,.7em>
\ar "a"+<-1.3em,-.7em>;"c"+<-1.3em,.7em>^{\hat{\pi}}
\ar "a"+<1.8em,-.2em>;"b"+<-1.em,-.2em>
\ar "c"+<1.8em,.em>;"d"+<-1.em,.em>
\ar "b"+<.em,-.7em>;"d"+<.em,.7em>_{\pi^\prime}
\ar "e"+<.em,-.8em>;"b"+<.em,.8em>
\ar@/^-1em/ "d"+<.8em,.5em>;"e"+<.8em,-.5em>_s
\ar@/^-.5em/ "c"+<1.6em,.5em>;"a"+<1.6em,-.5em>^{\hat s}
\ar@/^-.3em/ "d"+<.5em,.5em>;"b"+<.5em,-.5em>
\end{xy}
\]

\begin{proposition}\label{proper}
The complex manifold $\widehat{N}$ is a proper Kobayashi hyperbolic manifold.
\end{proposition} 
\begin{proof} 
Let $\widetilde{N}$ be the universal covering of $\widehat{N}$.
First, we see that $\widetilde{N}$ is a proper Kobayashi hyperbolic manifold.
Recall that $q^\prime$ is the holomorphic quadratic differential on  $Y=X/{\rm Ker}(D) $ which is the projection of the holomorphic quadratic differential $q$ on $X$.
By setting $\iota^\prime :\mathbb{H} \rightarrow T(Y)$ to be the holomorphic isometric embedding defined by $q^\prime$.
Then we can write $\widetilde{N}$ as
\begin{center}
$\widetilde{N}=\{(\tilde t, \tilde w): \tilde t \in \mathbb{H},\ \tilde w \in \mathbb{H}^{\tilde t}\}$
\end{center} 
by identifying $\mathbb{H}$ as $\iota^\prime(\mathbb{H})$.
And $\mathbb{H}^{\tilde t}$ is defined in subsection \ref{bfs}.
Therefore, $\widetilde{N}$ is a closed submanifold of the Bers fiber space $F(G^\prime)$.
Here, $G^\prime$ is a Fuchsian group with $Y=\mathbb{H}/G^\prime$.
Since $F(G^\prime)$ is proper Kobayashi hyperbolic manifold by theorem \ref{Royden} and theorem \ref{Bers}, $\widetilde{N}$ is also a proper Kobayashi hyperbolic manifold.

Next, we show that $\widehat{N}$ is a Kobayashi hyperbolic manifold. 
We suppose that $d_{\widehat{N}}^K((t_1, w_1), (t_2, w_2))=0$.
Since $\hat{\pi}$ is holomorphic, 
\begin{center}
$d_\mathbb{H}(t_1, t_2) \leq d_{\widehat{N}}^K((t_1, w_1), (t_2, w_2)) =0$.
\end{center}
Hence, we have $t_1=t_2$.
We set $t=t_1=t_2$.
Let $\rho  : \widetilde{N} \rightarrow \widehat{N}$ be the universal covering map.
Fix $(\tilde t, \tilde w_1) \in \rho^{-1}((t, w_1))$.
By theorem \ref{Kobayashi}, we have 
\begin{center}
$\displaystyle 0=d_{\widehat{N}}^K((t, w_1), (t, w_2))= \inf_{(\tilde t, \tilde w_2) } d_{\widetilde{N}}^K\left((\tilde t, \tilde w_1), (\tilde t, \tilde w_2)\right)$.
\end{center}
Here, infimum is taken over all points in $\rho^{-1}((t, w_2))$.
For all $n\geq 2$, there exist $(\tilde t, \tilde w_n)$ such that $d_{\widetilde{N}}^K\left((\tilde t, \tilde w_1), (\tilde t, \tilde w_n)\right)<\frac{1}{n}$.
Since $\widetilde{N}$ is holomorphically embedded in $\mathbb{H} \times (\mathbb{C}-\{0, 1\})$, we have
\begin{center}
$d_{\mathbb{C}-\{0, 1\}}^K(\tilde w_1, \tilde w_n) \leq d_{\widetilde{N}}^K\left((\tilde t, \tilde w_1), (\tilde t, \tilde w_n)\right)<\frac{1}{n}$.
\end{center}
Since the relative topology of $\mathbb{H}^t$ from $d_{\mathbb{C}-\{0, 1\}}^K$ is same as the topology defined by $d_{\mathbb{H}^t}^K$ and $w_1 \in \mathbb{H}^t$, we have $d_{\mathbb{H}^t}^K(\tilde w_1, \tilde w_n)\rightarrow 0$ as $n\rightarrow \infty$.
This implies that $\tilde w_n=\tilde w_{n+1}$ for sufficiently large $n$ and, hence, $\tilde w_n=\tilde w_1$.
Therefore, we have $(t_1, w_1)=(t_2, w_2)$.

Finally we prove that $\widehat{N}$ is a proper Kobayashi hyperbolic manifold.
Let $U_r(t_0, w_0)$ be the closed ball of center $(t_0, w_0) \in \widehat{N}$ and radius $r$ in $\widehat{N}$.
Take any sequence $\{(t_n, w_n)\}_{n \in \mathbb{N}}$ in $U_r(t_0, w_0)$.
We fix $(\tilde t_{0}, \tilde w_{0}) \in \rho^{-1}((t_{0}, w_{0}))$. 
By proposition \ref{Kobayashi}, there exists $(\tilde t_n, \tilde w_n) \in \rho^{-1}((t_n, w_n))$ such that 
$d_{\widetilde{N}}^K \left( (\tilde t_{0}, \tilde w_{0}), (\tilde t_n, \tilde w_n)\right)<r$ 
for each $n \in \mathbb{N}$.
Since $\widetilde{N}$ is proper, there exists a subsequence $\{(\tilde t_{n_l}, \tilde w_{n_l})\}$ which converges to  some point $(\tilde t_\infty, \tilde w_\infty) \in \widetilde{N}$.
Then,
\begin{eqnarray*}
d_{\widehat{N}}^K\left((t_{n_l}, w_{n_l}), \rho(\tilde t_\infty, \tilde w_\infty)\right) 
\leq
d_{\widetilde{N}}^K\left((\tilde t_{n_l}, \tilde w_{n_l}), (\tilde t_\infty, \tilde w_\infty)\right)
\rightarrow  0\ (l\rightarrow \infty). 
\nonumber
\end{eqnarray*}
Hence, $\widehat{N}$ is a proper Kobayashi hyperbolic manifold. 
\end{proof}

By proposition \ref{proper}, we can extend $\hat{s} : \mathbb{H}^* \rightarrow  \widehat{N}^*$ to a holomorphic section $\hat{s} : \mathbb{H} \rightarrow  \widehat{N}$ of $\hat{\pi}$.
For holomorphic sections of $\hat{\pi}$, we have the following. 
Recall that $q^\prime$  is a holomorphic quadratic differential on $Y=X/{\rm Ker}(D)$ induced by $q$ and
$g_t: Y \rightarrow Y_t$ is the Teichm\"uller map whose Beltrami coefficient is $\frac{i-t}{i+t}\frac{|q^\prime|}{q^\prime}$.
\begin{theorem}
Let $\hat{s}: \mathbb{H} \rightarrow \widehat{N}$ be a holomorphic section of $\hat{\pi}$.
Set $b=\hat{s}(i)$.
Then 
\begin{center}
$\hat{s}(t)= (t, g_t(b))$
\end{center}
for all $t \in \mathbb{H}$. 
 \end{theorem}
 
\begin{proof}
Set $\dot Y=Y-\{b\}$.
Let $\iota^\prime : \mathbb{H} \rightarrow T(Y) $ be the holomorphic isometric embedding defined by $q^\prime$.
Let $\tilde s: \mathbb{H} \rightarrow \widetilde{N}$ be a lift of $\hat{s}$ to $\widetilde{N}$.
Since, $\widetilde{N}$ is embedded holomorphically into the Bers fiber space $F(G^\prime)$, 
it is also embedded  holomorphically into $T(\dot Y)$. 
Here, $G^\prime$ is a Fuchsian group with $Y =  \mathbb{H} / G^\prime$.
The composition of this embedding with $\tilde s$ is an injective holomorphic map 
$\kappa : \mathbb{H} \rightarrow T(\dot Y)$ with $\kappa(i)=[\dot Y, id]$.
Let $\tau : T(\dot Y) \rightarrow T(Y)$ be the forgetful map, that is, $\tau ([f(\dot Y), f])=[f(Y), f]$ 
 for all $[f(\dot Y), f] \in T(\dot Y)$.
 Then we have the following commutative diagram:
\[
\begin{xy}
(0,0)="a"*{\mathbb{H}},
"a"+/r7em/="b"*{T(Y)},
"a"+/u5.em/="c"*{\widetilde{N}},
"c"+/r7em/="d"*{T(\dot Y)},
\ar "a"+<1.em,.em>;"b"+<-1.8em,.em>^{\iota^\prime}
\ar "c"+<1em,.em>;"d"+<-1.8em,.em>
\ar "c"+<.em,-.7em>;"a"+<.em,.7em>
\ar "d"+<.em,-.7em>;"b"+<.em,.7em>^{\tau}
\ar "a"+<.5em,.5em>;"d"+<-1.5em,-.5em>^{\kappa }
\ar@/^/ "a"+<-.8em,.5em>;"c"+<-.8em,-.5em>^{\widetilde s}
\end{xy}
\]
Since $\kappa , \tau$ is holomorphic and $\iota$ is holomorphic isometric embedding,
by theorem \ref{Royden}, $\kappa : \mathbb{H} \rightarrow T(\dot Y)$ is also an isometry.
By proposition \ref{embtohqd}, there exists a holomorphic quadratic differential $\dot q$ on $\dot Y$ which defines $\kappa$.
Take some $t\in \mathbb{H}$ $ (t \not=i)$ then the equation  $\tau \circ \kappa(t)=\iota^\prime(t)$ implies that  $\frac{i-t}{i+t}\frac{|\dot q|}{ \dot q}=\frac{i-t}{i+t}\frac{|q^\prime|}{q^\prime}$ on $\dot Y$.
This means that $\hat{s}(t)= (t, g_t(b))$ for all $t \in \mathbb{H}$.
\end{proof}

Now, we observe that  $N^*$ is a quotient manifold of $\widehat{N}^*$ by the action of the Veech group $\Gamma(X,u)$.
Let $u^\prime$ be the flat structure on $Y=X/{\rm Ker}(D)$ defined by $q^\prime$.
Since $\Gamma(X,u)={\rm Aff}^+(X,u)/{\rm Ker}(D)$, the Veech group $\Gamma(X, u)$ is considered as a subgroup of ${\rm Aff}^+(Y,u^\prime)$ and acts on $Y$.
We denote by $h_A^\prime$ each element $A$ of $\Gamma(X, u)$ if we consider $A$ as an element of ${\rm Aff}^+(Y,u^\prime)$. 
We extend the action of $\Gamma(X, u)$  onto $Y$ to an action onto $\widehat{N}$.
For all $(t, w) \in \widehat{N}$ and $h_{A}^\prime=A\in \Gamma(X, u) <{\rm Aff}^+(Y,u^\prime)$, we set
\begin{center}
$\left( h_A^\prime \right)_*(t,w) =(\bar A(t), g_{\bar A(t)}h_A^\prime g_t^{-1}(w))$.
\end{center}
Recall that $
R=\tiny{\left(
  \begin{array}{cccc}
   -1 & 0  \\
    0 & 1
  \end{array}
  \right)}$ and
$\bar A=RAR^{-1}$ acts on $\mathbb{H}$ as a
 M\"obius transformation. 
This action is induced by the action of ${\rm Aff}^+(X,u)$ on $\widehat{M}$ which we define in section \ref{vhf}.
Then, it is clear that $N^*=\widehat{N}^*/{\rm Ker}(D)$.
Now we have the following.
\begin{corollary}
Let $\varphi  : X \rightarrow Y$ be the branched covering map.
For $b \in Y$ and any point $a \in \varphi^{-1}(b)$, the following are equivalent.
\begin{itemize}
\item[(1)] The map
$\hat s_b: \mathbb{H} \ni t \mapsto(t, g_t(b))\in \widehat{N} $ induces a section of $\pi^\prime: N^* \rightarrow B$,
\item[(2)] $\Gamma(X,u)(\{b\})=\{b\}$,
\item[(3)] ${\rm Aff}^+(X,u)(\{a\})={\rm Ker}(D)(\{a\})$.
\end{itemize}
\end{corollary}
\begin{proof}
$(1)\Rightarrow (2)$. 
If $\hat s_b$  induces a section of $\pi^\prime$, then $(h_A^\prime)_*(\hat s_b(t))=\hat s_b(\bar A(t))$ for all $h_A^\prime=A \in \Gamma(X, u)$.
This implies $h_A^\prime(b)=b$ for all $A \in \Gamma(X, u)$.

$(2)\Rightarrow (1)$. 
If  $\Gamma(X,u)(\{b\})=\{b\}$ holds, then $(h_A^\prime)_*(s_b(t))=s_b(\bar A(t))$ for all $h_A^\prime=A \in \Gamma(X, u)$.
Hence, we obtain $(1)$.

Clearly, $(2)$ and $(3)$ are equivalent since $\Gamma(X, u)= {\rm Aff}^+(X, u)/{\rm Ker}(D)$.
 \end{proof}
Recall that $\bar{\rho}|_{\mathbb{H}^*}: \mathbb{H}^* \rightarrow B= \mathbb{H}^*/ \bar{\Gamma}(X,u)$ is a covering map and $f_{{t}}: X \rightarrow X_{{t}}$ is the Teichm\"uller map with the Beltrami coefficient $\frac{i-{{t}}}{i+{{t}}}\frac{|q|}{q}$ for each ${{t}} \in \mathbb{H}^*$.
Fix any ${t} \in \mathbb{H}^*$.
We take an open neighborhood $\widetilde{U}$ of $\tilde{t}$ in $\mathbb{H}^*$ such that $\bar{\rho}|_{\widetilde{U}}: \widetilde{U} \rightarrow  \bar{\rho}(\widetilde{U})=U$ is a conformal map.
Then $\pi^{-1}(U) \subset M$ is biholomorphic to $\widetilde{M}(\widetilde{U})=\{(t, z) : t \in \widetilde{U},\ z \in X_t\} \subset  \widetilde{M}$.
By identifying $U$ with $\widetilde{U}$ and $\pi^{-1}(U)$ with $\widetilde{M}(\widetilde{U})$, we have the following.
 \begin{corollary}\label{corollary1}
Let $s: B \rightarrow M$ be a holomorphic section of $\pi$.
Then there exists $a \in X$ with ${\rm Aff}^+(X,u)(\{a\})={\rm Ker}(D)(\{a\})$ such that the induced holomorphic map $s|_{\widetilde{U}}: \widetilde{U} \rightarrow \widetilde{M}(\widetilde{U})$ satisfies
\begin{center}
$s|_{\widetilde{U}}(t)=\left( t, f_t(a) \right)$
\end{center}
for all $t \in U$. 
\end{corollary}
Fix $\Gamma<\Gamma(X,u)$. 
Let $\rho^\prime: \mathbb{H}^*/\Gamma\rightarrow B$
be a covering map.
Set $M_\Gamma=\{(t^\prime,z^\prime): t^\prime \in \mathbb{H}^*/\Gamma, z^\prime \in X_{\rho^\prime(t^\prime)}=\pi^{-1}(\rho^\prime(t^\prime))\}$ and let $\pi_\Gamma: M_\Gamma \rightarrow \mathbb{H}^*/\Gamma$ be the projection.
Then the triple $(M_\Gamma, \pi_\Gamma, \mathbb{H}^*/\Gamma)$ is a holomorphic family of Riemann surfaces, $M_\Gamma$ is a covering space of $M$, and $\pi_\Gamma$ is a lift of $\pi$. 
By identifying $V=(\rho^\prime)^{-1}(U)$ with $\widetilde{U}$ and $\pi_\Gamma^{-1}(V)$ with $\widetilde{M}(\widetilde{U})$, we also have the following.
\begin{corollary}\label{corollary2}
For a holomorphic section $s: \mathbb{H}^*/\Gamma \rightarrow M_\Gamma$, there exists to a point $a \in X$ with $D^{-1}(\Gamma)(\{a\})={\rm Ker}(D)(\{a\})$ such that the induced holomorphic map $s|_{\widetilde U} : \widetilde{U} \rightarrow \widetilde{M}(\widetilde{U})$
satisfies
\begin{center}
$s|_{\widetilde U}(t^\prime)=\left ( t^\prime, f_{t^\prime}(a) \right)$
\end{center}
for all $t^\prime \in \widetilde{U}$.
\end{corollary}
\begin{remark}
Note that it is not true that all points $a$ of $X$ with ${\rm Aff}^+(X,u)(\{a\})={\rm Ker}(D)(\{a\})$ define holomorphic sections of $\pi: M \rightarrow B$. 
See example \ref{D-example}.
\end{remark}
\begin{example}
Let $(X, u)$ be the flat surface as in example \ref{example1}.
Recall that $(X, u)$ is a Riemann surface of type $(2, 0)$ and $\Gamma(X, u)={\tiny \left< {  \left[\left(
  \begin{array}{cccc}
    1 & 1 \\
    0 & 1 
  \end{array}
  \right)
  \right]}, 
{   \left[\left(
  \begin{array}{cccc}
    1 & 0 \\
    2 & 1
  \end{array}
  \right)
  \right] }
\right>}$.
It is easy to see that ${\rm Ker}(D)=\left<\alpha : z \mapsto z+1, \beta: z \mapsto -z \right>$.
The candidates of holomorphic sections of the corresponding Veech holomorphic family of Riemann surfaces are the points $a$ with ${\rm Aff}^+(X,u)(\{a\})={\rm Ker}(D)(\{a\})$.
The candidates are $(0, 0)$, $(\frac{1}{2}, 0)$, $(\frac{3}{2}, 0)$, $(0, 1)$, $(\frac{1}{2}, 1)$, $(\frac{3}{2}, 1)$ in figure \ref{Veechexam1-1}.
They are the Weierstrass points of $X$ and the hyperelliptic involution of $X$ is $\alpha^2$.
\end{example}

\begin{example}
Let $(X, u)$ be the flat surface as in example \ref{example2}.
Then $\Gamma(X, u)=\tiny \left< R_n, T_n \right>$.
It is easy to see that ${\rm Ker}(D)=\left<z \mapsto -z \right>$.
Only the points of $X$ which correspond to the center and the vertices of the corresponding polygon satisfy
${\rm Aff}^+(X,u)(\{a\})={\rm Ker}(D)(\{a\})$.
Hence, the two points are the candidates of the holomorphic sections of the corresponding Veech holomorphic family of Riemann surfaces.
Since, one of them is the critical point of $(X, u)$ and the other is not a critical point of $(X, u)$, if we construct sections locally  from the points and extend them globally, the permutation of these two points does not happen.
Hence, the two points define holomorphic sections and they are all the holomorphic sections of the Veech holomorphic family of Riemann surfaces. 
\end{example}
 \section{Diophantine problems}\label{Diophantine}
In this section, we consider Diophantine problems on function fields. 
We associate the Diophantine problems with holomorphic families of Riemann surfaces and their holomorphic sections.
At the end of this section, we construct some examples of Diophantine equations on function fields and give all solutions of them. 
 
Let $B_0$ be a compact Riemann surface. 
Denote by $M(B_0)$ the field of all meromorphic functions on $B_0$.
We take an irreducible homogeneous polynomial
\begin{center}
$\displaystyle f(X,Y,Z)= \sum_{i+j+k=N} A_{i,j,k} X^iY^jZ^k$ $(A_{i,j,k}\in M(B_0))$.
\end{center} 
\begin{problem}[Diophantine problem]
Find all solutions of $f(X,Y,Z)=0$ in $\mathbb{P}^2(M(B_0))$.
\end{problem}
We see a relation between Diophantine problems and holomorphic families of Riemann surfaces.
We set
\begin{eqnarray}
f_t(x,y,z)= \sum_{i+j+k=N} A_{i,j,k}(t) x^iy^jz^k.
\nonumber
\end{eqnarray} 
Since $f$ is irreducible, there exist $t_1, \cdots, t_n \in B_0$ such that $f_t(x,y,z)\in \mathbb{C}[x, y,z]$
is irreducible for each $t \in B=B_0-\{t_1, \cdots, t_n\}$.
Then, for each $t \in B$, 
\begin{center}
$X_t=\{[x:y:z] \in \mathbb{P}^2(\mathbb{C}): f_t(x,y,z)=0\}$
\end{center}
is an algebraic curve, that is, a compact Riemann surface.
Set 
\begin {center}
$\displaystyle M=\{(t, [x:y:z]) \in B \times \mathbb{P}^2(\mathbb{C}) : f_t(x,y,z)=0\}
$
\end{center}
and let $\pi : M \rightarrow B$ be the projection.
Then the triple $(M, \pi, B)$ is a holomorphic family of Riemann surfaces. 
Next, if $[X: Y: Z] \in \mathbb{P}^2(M(B_0))$ is a solution of $f(X,Y,Z)=0$, then 
\begin{eqnarray}
f(X(t),Y(t),Z(t))= \sum_{i+j+k=N} A_{i,j,k}(t) X(t)^iY(t)^jZ(t)^k=0
\nonumber
\end{eqnarray} 
for all $t \in B_0$.
This means that $[X(t) : Y(t) : Z(t)] \in X_t$ for all $t \in B$.
Hence, the map $B \ni t \mapsto (t, [X(t): Y(t): Z(t)]) \in M$ is a holomorphic section of $\pi: M \rightarrow B$.
Conversely, if $s: B \rightarrow M$ is a holomorphic section, then $s(t)=(t, [X(t): Y(t): Z(t)])$ for some $X, Y, Z \in M(B_0)$ and $[X: Y: Z]$ is a solution of $f(X,Y,Z)=0$.
In conclusion, a Diophantine problem is to find all holomorphic sections of the corresponding holomorphic families of Riemann surfaces.
\begin{example}\label{D-example}
Let $(X, u)$ be the flat surface as in example \ref{example1}.
Let $(M, \pi, B)$ be the Veech holomorphic family of Riemann surfaces defined by $(X,u)$.
Set $S={\tiny  \left[\left(
  \begin{array}{cccc}
    1 & 1 \\
    0 & 1 
  \end{array}
  \right)
  \right]}$ 
and
$T={\tiny   \left[\left(
  \begin{array}{cccc}
    1 & 0 \\
    2 & 1
  \end{array}
  \right)
  \right] }$.
Then $X$ is a Riemann surface of type $(2,0)$, $\Gamma(X, u)= \left< S,T \right>$, ${\rm Ker}(D)=\left<\alpha : z \mapsto z+1, \beta: z \mapsto -z \right>$, $\alpha^2$ is the hyperelliptic involution of $X$, and the Weierstrass points of $X$ are $(0, 0)$, $(\frac{1}{2}, 0)$, $(\frac{3}{2}, 0)$, $(0, 1)$, $(\frac{1}{2}, 1)$ and $(\frac{3}{2}, 1)$.
We denote the set of all Weierstrass points of $X$ by $WP$.
Since $\mathbb{H}/\Gamma(X,u)$ is genus $0$ and has $2$ punctures and one cone point of order $2$, $B=\mathbb{H}^*/\bar{\Gamma}(X,u)$ is a 3-punctured sphere. 
Recall that 
$
R=\tiny{\left(
  \begin{array}{cccc}
   -1 & 0  \\
    0 & 1
  \end{array}
  \right)}$ 
  and $\bar{\Gamma}(X,u)=R \Gamma(X,u) R^{-1}$.

We describe each fiber $X_t=\pi^{-1}(t)$ $(t\in B)$ as an algebraic curve.
By construction, $X_t=(X, A_t\circ u)$ for some $A_t \in {\rm SL}(2, \mathbb{R})$.
Let $T_{A_t}: (X,u) \rightarrow (X, A_t \circ u)$ be the identity map.
Note that $T_{A_t}$ is an affine map whose derivative is $A_t$ with respect to the charts of the flat structures.  
Set $\alpha_t=T_{A_t}\circ \alpha \circ T_{A_t}^{-1}$, $\beta_t=T_{A_t}\circ \beta \circ T_{A_t}^{-1}$ and ${WP}_t=T_{A_t}(WP)$.
Then  $K_t=\left< \alpha_t, \beta_t \right>$ is a group of conformal automorphisms of $X_t$.
The automorphism $\alpha_t^2$ is also the hyperelliptic involution of $X_t$ and ${WP}_t$ is the set of all  Weierstrass points of $X_t$.
Then $R_t=(X_t-{WP}_t)/\left< \alpha_t^2 \right>$ is a Riemann surface of type $(0,6)$ and has conformal automorphisms $\alpha_t^\prime, \beta_t^\prime$ induced by $\alpha_t, \beta_t$. 
We may assume that $0, 1, \infty$ are punctures of $Y_t$, two critical points of $X_t$ are mapped to $0, \infty $ and another Weierstrass point is mapped to $1$. 
Then $\alpha_t^\prime$ is a M\"obius transformation of order 2 which fixes $0, \infty$ and $\beta_t^\prime$ is a M\"obius transformation of order 2 which permutes $0$ and $\infty$.
Therefore, $\alpha_t^\prime(z)=-z$ and $\beta_t^\prime(z)=\lambda_t/z$ for some $\lambda_t \in \mathbb{C}-\{0,\pm 1\}$.
This implies that $R_t= \mathbb{C}-\{0,\pm 1,\pm \lambda_t\}$ and 
\begin{center}
$X_t=\{(x, y) \in \hat{\mathbb{C}}^2: y^2=x(x^2-1)(x^2-\lambda_t^2)\}$.
\end{center}
Also, we can write $X_t$ as 
\begin{center}
$X_t=\{(x, y) \in \hat{\mathbb{C}}^2: y^2=x(x^2-1)(x^2-1/\lambda_t^2)\}$.
\end{center}
Thus, we have the holomorphic map 
\begin{center}
$\Psi  : \mathbb{H}/\bar{\Gamma}(X, u) \ni t \mapsto \{\lambda_t^2, 1/\lambda_t^2\} \in \left(\mathbb{C}-\{0,1\}\right)/\left<z \mapsto 1/z \right>$.
\end{center}
By construction, $X_t$ and $X_{t^\prime}$ are not conformal equivalent for two distinct points $t, t^\prime \in \mathbb{H}/\bar{\Gamma}(X, u)$.
This implies that $\Psi$ is injective, and hence conformal map. 
The map $\phi (z)=\frac{1}{2}(z+1/z)$ induces a conformal map from $\left(\mathbb{C}-\{0,1\}\right)/\left<z \mapsto 1/z \right>$ to $\mathbb{C}-\{0, 1\}$.
And,  $-1 \in \mathbb{C}-\{0, 1\}$ is a cone point of $\left(\mathbb{C}-\{0,1\}\right)/\left<z \mapsto 1/z \right>$.
Hence, $B=\mathbb{H}^*/\bar{\Gamma}(X, u)=\mathbb{C}-\{0, \pm 1\}$.
Since the parameter $t$ is in $B$ and $\lambda_t^2, 1/\lambda_t^2$ are in $\mathbb{C}-\{0,\pm1\}=\phi^{-1}(B)$ which is a double cover of $B$, if the parameter $t$ moves on $B$ along a closed curve around $1$ or $-1$, then the descriptions of $X_t$ as above appear reciprocally.
(See figure \ref{Diophantine4}.)
We change the base space $B$ to $\phi^{-1}(B)=\mathbb{C}-\{0,\pm 1\}$.
By calculation, it is easily proved that $\Gamma=\left< S, T^2, TST^{-1} \right>$ is the subgroup of $\Gamma(X, u)$ which corresponds to the covering map $\phi : \mathbb{C}-\{0,\pm 1\}\rightarrow B$.
Let $(M_\Gamma, \pi_\Gamma, \mathbb{C}-\{0,\pm 1\})$ be the holomorphic family of Riemann surfaces induced by $(M, \pi, B)$ by replacing $B$ to $B_\Gamma=\mathbb{H}^*/R \Gamma R^{-1}=\mathbb{C}-\{0,\pm 1\}$.
Then
\begin {eqnarray}
\displaystyle M_\Gamma &=&\{(t, (x,y)) \in B_\Gamma \times \hat{\mathbb{C}}^2 : 
y^2=x(x^2-1)(x^2-t)\}
\nonumber
\\
&=&\{(t, [X:Y:Z]) \in B_\Gamma \times \mathbb{P}^2(\mathbb{C}) : 
Y^2Z^3=X(X^2-Z^2)(X^2-tZ^2)\}.
\nonumber
\end{eqnarray}
From corollary \ref{corollary2}, all holomorphic sections of $\pi_\Gamma$ correspond to the points in $X$ which satisfy 
$D^{-1}(\Gamma)(\{a\})={\rm Ker}(D)(\{a\})$.
By computation, it is easy to see that such points are exactly the Weierstrass points.
Hence, the candidates of holomorphic sections are
\begin{eqnarray}
(t,[X:Y:Z])&=& (t,[0:0:1]), (t,[1:0:1]),(t,[-1:0:1]), \nonumber \\
&&(t,[0:1:0]),(t,[\sqrt{t}:0:1]), (t,[-\sqrt{t}:0:1]).
\nonumber
\end{eqnarray}
They are well-defined except for the last two.
Thus, we conclude that all the holomorphic sections of $\pi_\Gamma$ are  
\begin{center}
$(t,[X:Y:Z])= (t,[0:0:1]), (t,[1:0:1]),(t,[-1:0:1]), (t,[0:1:0])$
\end{center}
and all the solutions of the Diophantine equation $f(X, Y, Z)=X(X^2-Z^2)(X^2-tZ^2)-Y^2Z^3$ 
in $\mathbb{P}^2(M(\hat{\mathbb{C}}))$ is  
\begin{center}
$[X:Y:Z]= [0:0:1], [1:0:1],[-1:0:1], [0:1:0]$.
\end{center}
\begin{figure}[h]
 \begin{center}
 \includegraphics*[keepaspectratio, scale=1]{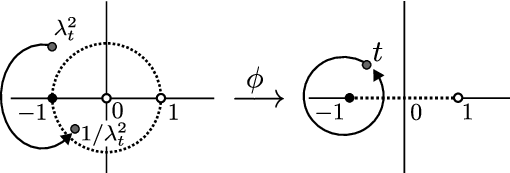}
\caption{A closed curve with base point $t$ and a lift of the curve}
\label{Diophantine4}
 \end{center}
 \end{figure}
\end{example}
\begin{remark}
We can extend $(M_\Gamma, \pi_\Gamma, \mathbb{C}-\{0,\pm 1\})$ to a holomorphic family of Riemann surfaces over $\mathbb{C}-\{0,1\}$.
On the other hand, $t=0,\infty$ correspond to the Riemann surface with one node obtained by contracting a horizontal closed curve on $(X,u)$.
And $t=1$ corresponds to a Riemann surface with two nodes obtained by contracting vertical closed curves on $(X,u) $ which are the core curves of two vertical cylinders.
\end{remark}
\begin{example}\label{D-example2}
By the same argument as example \ref{D-example}, we obtain the following examples.
Let us consider the flat surface $(X_m, u_m)$ as in figure \ref{Diophantine5} for $m \geq 2$.
If $m$ is even, $X_m$ has genus $m$ and two critical points.
And if $m$ is odd, $X_m$ has genus $m-1$ and four critical points.
The Veech groups $\Gamma(X_m,u_m)$ are equal to $\Gamma(X,u)=\left<S, T\right>$ as in example \ref{D-example}. 
Let $(M_m, \pi_m, \mathbb{C}-\{0,\pm 1\})$  be the holomorphic families of Riemann surfaces constructed in the same way as example \ref{D-example}.
Then $(M_m, \pi_m, \mathbb{C}-\{0,\pm 1\})$  corresponds to a Diophantine equation
\begin{center}
$f_m(X,Y,Z)=X(X^m -Z^m)(X^m -t Z^m)-Y^2 Z^{2 m-1}=0$.
\end{center} 
And all solutions of $f_m(X,Y,Z)=0$ are
\begin{center}
$[X:Y:Z]= [0:0:1], [0:1:0], [{\rm exp}(2\pi i \nu /m):0:1]$ $(\nu =1,2,\cdots ,m)$.
\end{center} 
\begin{figure}[h]
 \begin{center}
 \includegraphics*[keepaspectratio, scale=1]{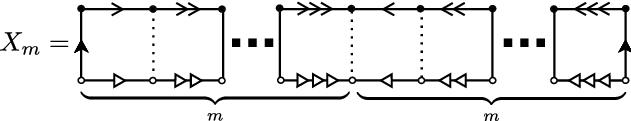}
\caption{The flat surface $X_m$}
\label{Diophantine5}
 \end{center}
 \end{figure}

\end{example}

 \section{Upper bound of the number of the holomorphic sections}\label{number}
 In this section, we give an upper bound of the numbers of holomorphic sections of certain Veech holomorphic families of Riemann surfaces by topological quantities of $B$ and the fiber Riemann surfaces. 
 
Let $X$ be a Riemann surface of type $(g, n)$ and $q$ be an integrable holomorphic quadratic differential on $X$.
Denote by $u$ the flat structure on $X$ induced by $q$.
Let $(M, \pi, B)$ be the Veech holomorphic family of Riemann surfaces defined by $(X,u)$.
We assume that $\Gamma(X,u)$ is a Fuchsian group of type $(p, k: \nu_1, \cdots, \nu_k)$ $(\nu_i \in \{2,3, \cdots, \infty\})$.
Then $B=\mathbb{H}^*/\Gamma(X, u)$ is a Riemann surface of type $(p,k)$. 
We prove the following theorem.
\begin{theorem}\label{numberthm}
If $q$ has a simple Jenkins-Strebel direction, then the number of holomorphic sections of $(M, \pi, B)$ is at most
\begin{center}
$32\pi (2p-2+k)(3g-3+n)^2(3g-2+n)-2g+2$.
\end{center}
\end{theorem} 
From corollary \ref{corollary1}, the number of holomorphic sections is at most the cardinal of the set\begin{center}
$S(X, u)=\{a \in X: {\rm Aff}^+(X, u)(\{a\})={\rm Ker}(D)(\{a\})\}$.
\end{center}
We estimate the cardinal of $S(X, u)$.
We may assume that $\theta=0$ is a simple Jenkins-Strebel direction of $q$.
Let  $C(X, u)$ be the set of all critical points on $X$.
Let $\varphi: X \rightarrow Y=X/{\rm Ker}(D) $ be the branched covering map and $B(X)$ be the set of all branch points of $\varphi$.
The Riemann surface $Y= X/{\rm Ker}(D)$ has the flat structure $u^\prime$ induced by $(X, u)$.
We denote the set of all critical points of $(Y, u^\prime)$ by $C(Y, u^\prime)$ and the holomorphic quadratic differential on $Y$ which corresponds to $(Y, u^\prime)$ by $q^\prime$. 
Note that $\varphi(C(X, u))=C(Y, u^\prime)$ is not true.
Recall that the Veech group $\Gamma(X,u)$ corresponds to a subgroup of the affine group ${\rm Aff}^+(Y, u^\prime)$ of $(Y, u^\prime)$.
If we consider $\Gamma(X, u)$ as a subgroup of ${\rm Aff}^+(Y, u^\prime)$, we write $A \in \Gamma(X, u)$ by $h_A^\prime$.
Clearly, $\varphi(S(X, u))=\{b \in Y: \Gamma(X, u)\{b\}=\{b\}\}$.

Since  $\theta=0$ is a simple Jenkins-Strebel direction, $(X, u)$ is constructed from a rectangle 
whose vertical sides are glued by a translation and horizontal sides correspond to the union of all horizontal saddle connections of $(X, u)$.
Let $H, W >0$ be the height and the width of the rectangle, respectively.
Let  $R$ be  the domain in $X$ which corresponds to the interior of the rectangle.
Then there exists a chart  $z: R \rightarrow (0,W) \times (-\frac{H}{2}, \frac{H}{2})$ in $u$.
For every $h \in {\rm Ker}(D)$, there exists  an open set $U \subset R$ such that $h(U) \subset R$
and $z \circ h \circ z^{-1}: z(U)\rightarrow z(h(U))$ is the form $z \mapsto \pm z+c$ for some $c \in \mathbb{R}$.
The sign of the representation of $h$ is uniquely determined  and $c$ is also uniquely determined up to $\pm W$.
We identify $h$ with the representation.
Then the following lemma is obvious.
\begin{lemma}\label{kernel}
The group ${\rm Ker}(D)$ is one of the following form:
\begin{center}
$\left< z \mapsto z+c\right>, \left<z \mapsto -z+c \right>, \left<z \mapsto z+c,  z \mapsto -z+d \right>$.
\end{center}
\end{lemma}
Let $sgn : {\rm Ker}(D) \rightarrow \{\pm 1\}$ be the homomorphism which maps $h \in {\rm Ker}(D)$ to the derivative of $z \circ h \circ z^{-1}$. 
Denote by $l_y$ the closed curve in $X$ with $z(R \cap l_y)=(0,W) \times \{y\}$ for $y\in  (-\frac{H}{2}, \frac{H}{2})$.
\begin{lemma}\label{involution}
Every $h \in {\rm Ker}(D)$ with $sgn(h)=-1$ has two fixed points on $l_0$ and $h^2=id$. 
\end{lemma}
\begin{proof}
Since $h$ fixes $l_0$ and reverse the orientation of $l_0$, $h$ has two fixed points in $l_0$.
This implies that $z \circ h^2 \circ z^{-1}$ is a translation with fixed points in $l_0$.
Hence, we have $h^2=id$.
\end{proof}
We set
\begin{eqnarray}
b_0&=& \inf \Bigl\{|b| : 
\tiny  \left[\left(
  \begin{array}{cccc}
    1 & b \\
    0 & 1 
  \end{array}
  \right)
  \right] \in \Gamma(X, u),\ b \not =0
\Bigr\},
\nonumber \\
c_0&=& \inf \Bigl\{|c| : 
\tiny  \left[\left(
  \begin{array}{cccc}
    a & b \\
    c & d 
  \end{array}
  \right)
  \right] \in \Gamma(X, u),\ c \not =0
\Bigr\}.
\nonumber
\end{eqnarray}
To prove theorem \ref{numberthm}, we show the following.
\begin{proposition}\label{number2}
If $sgn({\rm Ker}(D))=\{1\}$, then
 \begin{center}
$\sharp S(X, u) \leq {\rm mod}(X)c_0 \lceil \frac{b_0 \sharp {\rm Ker}(D)}{{\rm mod}(X)}+1 \rceil -2g+2$.
 \end{center}
 And if $sgn({\rm Ker}(D)=\{\pm 1\}$, then
 \begin{center}
$\sharp S(X, u) \leq 4{\rm mod}(X)c_0 \lceil$ $\frac{b_0 \sharp {\rm Ker}(D)}{4{\rm mod}(X)}+\frac{1}{2} \rceil -2g+2$.
\end{center}
Here, ${\rm mod}(X)=\frac{W}{H}$ and $\lceil x \rceil$ is the smallest integer which is greater than or equals to $x$. 
\end{proposition}
By lemmas \ref{kernel} and \ref{involution}, $(Y, u^\prime)$ is also constructed  from some rectangle  $R^\prime$ whose vertical sides are glued by translation and horizontal sides correspond to the union of all horizontal saddle connections of $(Y, u^\prime)$.
Denote by $W^\prime$ and $H^\prime$ the width and height of  $R^\prime$, respectively.
If $sgn({\rm Ker}(D))=\{1\}$, we have  $W^\prime = \frac{W}{\sharp{\rm Ker}(D)}$ and $H^\prime =H$.
If $sgn({\rm Ker}(D))=\{\pm 1\}$, we have $W^\prime=\frac{2W}{\sharp{\rm Ker}(D)}$ and $H^\prime=\frac{H}{2}$.
Note that, if $sgn({\rm Ker}(D))=\{\pm 1\}$, the closed curve $l_0$ in $X$ projects to one of horizontal sides of  $R^\prime$.
The horizontal side has two critical points which are poles of $q^\prime$ of order 1. 
We fix a vertical 
open interval $I$ in $R^\prime$ which connects two horizontal sides of  $R^\prime$.
Let $l^\prime_w$ be the horizontal closed geodesic in $Y$ passing through $w \in I$
and $L^\prime$ be the union of all horizontal saddle connections of $Y$. 
We set $h_{A_0}^\prime=A_0= {\tiny \left[\left(
  \begin{array}{cccc}
    1 & b_0 \\
    0 & 1 
  \end{array}
  \right)
  \right]} \in \Gamma(X, u)$,
$I_0=\{w \in I: l_w^\prime$ is pointwise fixed by $h_{A_0}^\prime \}$,
and 
\begin{center}
$\displaystyle {\rm Cross}(A)=\left(\bigcup_{w \in I_0} \left( l_w^\prime \cap h_A^\prime(l_w^\prime) \right) \right)\cup \left( L^\prime \cap h_A^\prime(L^\prime)\right)$
\end{center}
for $A= {\tiny \left[\left(
  \begin{array}{cccc}
    a & b \\
    c & d 
  \end{array}
  \right)
  \right]} \in \Gamma(X, u)$ with $c \not =0$.
\begin{lemma}\label{subset}
The sets $\varphi(C(X, u)), C(Y, u^\prime), \varphi(B(X))$ and $\varphi(S(X, u))$ are contained in ${\rm Cross}(A)$. 
\end{lemma}
\begin{proof}
Let $h_A \in {\rm Aff}^+(X, u)$ such that $D(h_A)=A$.
Since $h_A(C(X, u))=C(X, u)$, we have $h_A^\prime(\varphi(C(X, u)))=\varphi(C(X, u))$.
And $h_A^\prime(C(Y, u^\prime))=C(Y, u^\prime)$ because $h_A^\prime \in {\rm Aff}^+(Y, u^\prime)$.
Also, $h_A(B(X))=B(X)$ implies that $h_A^\prime(\varphi(B(X))=\varphi(B(X)))$.
As $\varphi(C(X, u)), C(Y, u^\prime)$ and $\varphi(B(X))$ are subsets of $L^\prime$,  they are contained in $L^\prime \cap h_A^\prime(L^\prime)$.
 Therefore, $\varphi(C(X, u)), C(Y, u^\prime), \varphi(B(X)) \subset {\rm Cross}(A)$.
 By the definition of $S(X, u)$, it is clear that $\varphi(S(X, u)) \subset {\rm Cross}(A)$.
\end{proof}
\begin{lemma}\label{intersection number}
For all $w \in I$,
\begin{eqnarray}
\sharp(l_w^\prime \cap h_A^\prime(l_w^\prime))&=& \left\{
\begin{array}{ll}
\frac{{\rm mod}(X)}{\sharp {\rm Ker}(D)}  |c|  & \left( sgn({\rm Ker}(D))=\{1\} \right),\\[1em]
\frac{4{\rm mod}(X)}{\sharp {\rm Ker}(D)}  |c| &\left( sgn({\rm Ker}(D))=\{\pm 1\} \right).
\end{array}
\right.
\nonumber 
\end{eqnarray}
And,
\begin{eqnarray}
\sharp\left( L^\prime \cap h_A^\prime(L^\prime)\right)&=& \left\{
\begin{array}{ll}
\frac{{\rm mod}(X)}{\sharp {\rm Ker}(D)}  |c| -2g^\prime+2 & \left( sgn({\rm Ker}(D))=\{1\} \right),\\[1em]
\frac{4{\rm mod}(X)}{\sharp {\rm Ker}(D)}  |c| -2g^\prime+2&\left( sgn({\rm Ker}(D))=\{\pm 1\} \right).
\end{array}
\right.
\nonumber
\end{eqnarray}
Here, $g^\prime$ is the genus of $Y$.
\end{lemma}
\begin{proof}
We identify $l_w^\prime$ with the vector $v= {\tiny \left(
  \begin{array}{c}
    W^\prime\\
    0
  \end{array}
  \right)
  }$.
Then $h_A(l_w^\prime)$ is identified with 
$A v = {\tiny \left(
  \begin{array}{c}
    W^\prime a\\
    W^\prime c
  \end{array}
  \right)
  }$.
The closed curve $h_A(l_w^\prime)$ pass through $l_w^\prime$ whenever the height moves $H^\prime$.
Therefore, 
\begin{eqnarray}
\sharp(l_w^\prime \cap h_A^\prime(l_w^\prime))=W^\prime c/H^\prime= \left\{
\begin{array}{ll}
\frac{{\rm mod}(X)}{\sharp {\rm Ker}(D)}  |c|  & \left( sgn({\rm Ker}(D))=\{1\} \right),\\[1em]
\frac{4{\rm mod}(X)}{\sharp {\rm Ker}(D)}  |c| &\left( sgn({\rm Ker}(D))=\{\pm 1\} \right).
\end{array}
\right.
\nonumber 
\end{eqnarray}
For each $p^\prime \in C(Y, u^\prime)$, let ${\rm arg}(p^\prime)$ be the angle around $p^\prime$.
Then ${\rm arg}(p^\prime)=n \pi$ for some $n \in \mathbb{N}$ and $p^\prime$ appears $n$ times on the horizontal sides of the rectangle $R^\prime$.
Hence, 
\begin{eqnarray}
\displaystyle
\sharp(\left( L^\prime \cap h_A^\prime(L^\prime)\right))&=&\frac{1}{2}\left(2\cdot W^\prime c/H^\prime -  \sum_{p^\prime \in C(Y, u^\prime)} \frac{{\rm arg}(p^\prime)}{\pi} \right)+\sharp C(Y, u^\prime) \nonumber\\
&=& W^\prime c/H^\prime - \frac{1}{2}\sum_{p^\prime \in C(Y, u^\prime)} \left(\frac{{\rm arg}(p^\prime)}{\pi}-2 \right) \nonumber\\
&=& W^\prime c/H^\prime - \frac{1}{2}\sum_{p^\prime \in C(Y, u^\prime)} {\rm ord}(p^\prime) \nonumber\\
&=& \left\{
\begin{array}{ll}
\frac{{\rm mod}(X)}{\sharp {\rm Ker}(D)}  |c| -2g^\prime+2 & \left( sgn({\rm Ker}(D))=\{1\} \right),\\[1em]
\frac{4{\rm mod}(X)}{\sharp {\rm Ker}(D)}  |c| -2g^\prime+2&\left( sgn({\rm Ker}(D))=\{\pm 1\} \right).
\end{array}
\right.
\nonumber
\end{eqnarray}
Recall that ${\rm ord}(w)$ is the order of $w$ with respect to $q^\prime$.
\end{proof}
\begin{lemma}\label{interval}
\begin{eqnarray}
\sharp I_0 \leq \left\{
\begin{array}{ll}
{\lceil} \frac{b_0 \sharp {\rm Ker}(D)}{{\rm mod}(X)}{\rceil}& \left( sgn({\rm Ker}(D))=\{1\} \right),
\\[1em]
{\lceil} \frac{b_0 \sharp {\rm Ker}(D)}{4{\rm mod}(X)}+\frac{1}{2} {\rceil}-1& \left( sgn({\rm Ker}(D))=\{\pm 1\} \right).
\end{array}
\right.
\nonumber 
\end{eqnarray}
\end{lemma}
\begin{proof}
On the rectangle $R^\prime$, we can represent $h_{A_0}^\prime$ as the following form:
\begin{center}
$h_{A_0}\left(
  \begin{array}{cccc}
    x \\
    y 
  \end{array}
  \right)= \pm \left(
  \begin{array}{cccc}
    1 & b_0 \\
    0 & 1 
  \end{array}
  \right)
 \left(
  \begin{array}{cccc}
    x \\
    y 
  \end{array}
  \right) +
  \left(
  \begin{array}{cccc}
    W^\prime \xi  \\
    0 
  \end{array}
  \right)$
\end{center}
for some $0 \leq \xi<1$.
If the representation holds by minus, $\sharp I_0=1$.
If it holds by plus and $sgn({\rm Ker}(D))=\{1\}$, 
 by computation, we have
\begin{center}
$\sharp I_0 \leq 
{\lceil} \frac{b_0 \sharp {\rm Ker}(D)}{{\rm mod}(X)} {\rceil}$.
\end{center}
If it holds by plus and $sgn({\rm Ker}(D))=\{\pm 1\}$, 
$h_{A_0}^\prime$ maps each horizontal sides of $R^\prime$ to itself.
And one of the sides is the image of the horizontal closed geodesic $l_0$ and it has exactly two points of $\varphi(B(X))$.
Moreover, the length of horizontal saddle connection on $Y$ which connects the two points is $\frac{W^\prime}{2}$.
Hence, $\xi=0$ or $\frac{1}{2}$. 
By computation, we have
\begin{center}
$\sharp I_0 \leq 
{\lceil} \frac{b_0 \sharp {\rm Ker}(D)}{4{\rm mod}(X)}+\frac{1}{2} {\rceil}-1$.
\end{center}
\end{proof}
\begin{proof}[Proof of proposition \ref{number2}]
By lemma \ref{intersection number} and lemma \ref{interval}, 
we have
\begin{eqnarray}
\sharp {\rm Cross}(A) \leq \left\{
\begin{array}{ll}
\frac{{\rm mod}(X)}{\sharp {\rm Ker}(D)} |c| \left({\lceil} \frac{b_0 \sharp {\rm Ker}(D)}{{\rm mod}(X)}+1{\rceil}\right)-2g^\prime+2& \left( sgn({\rm Ker}(D))=\{1\} \right),
\\[1em]
\frac{4{\rm mod}(X)}{\sharp {\rm Ker}(D)}|c|\left({\lceil} \frac{b_0 \sharp {\rm Ker}(D)}{4{\rm mod}(X)}+\frac{1}{2} {\rceil}\right)-2g^\prime+2& \left( sgn({\rm Ker}(D))=\{\pm 1\} \right).
\end{array}
\right.
\nonumber 
\end{eqnarray}
By lemma \ref{subset},  
\begin{center}
$\sharp S(X, u) \leq \sharp\varphi^{-1}({\rm Cross}(A))=\left( \sharp{\rm Cross}(A)- \sharp \varphi\left(B(X)\right)\right)\cdot \sharp {\rm Ker}(D) +\sharp B(X)$.
\end{center}
Therefore, if ${\rm Ker}(D)=\{1\}$,
\begin{eqnarray}
\sharp S(X, u)&\leq& {\rm mod}(X)|c| \left({\lceil} \frac{b_0 \sharp {\rm Ker}(D)}{{\rm mod}(X)}+1{\rceil}\right)
\nonumber \\
&&+\left(-2g^\prime+2-\sharp \varphi\left(B(X)\right)\right)\cdot \sharp {\rm Ker}(D)+\sharp B(X).
\nonumber
\end{eqnarray}
And, if ${\rm Ker}(D)=\{\pm 1\}$,
\begin{eqnarray}
\sharp S(X, u)&\leq& 4{\rm mod}(X)|c| \left({\lceil} \frac{b_0 \sharp {\rm Ker}(D)}{4{\rm mod}(X)}+\frac{1}{2}{\rceil}\right)
\nonumber \\
&&+\left(-2g^\prime+2-\sharp \varphi\left(B(X)\right)\right)\cdot \sharp {\rm Ker}(D)+\sharp B(X).
\nonumber
\end{eqnarray}
Moreover,
\begin{eqnarray}
&&\left(-2g^\prime+2-\sharp \varphi\left(B(X)\right)\right)\cdot \sharp {\rm Ker}(D)+\sharp B(X) \nonumber\\
&=& \left(-2g^\prime+2\right) \cdot \sharp{\rm Ker}(D)- \sum_{w \in \varphi(B(X))}\left(\sharp {\rm Ker}(D)-\sharp \varphi^{-1}(w) \right)\nonumber\\
&=& \left(-2g^\prime+2\right) \cdot \sharp{\rm Ker}(D)- \sum_{w \in \varphi(B(X))} \sum_{z \in \varphi^{-1}(w)}\left(e_z-1 \right)\nonumber\\
&=& \left(-2g^\prime+2\right) \cdot \sharp{\rm Ker}(D)- \sum_{z \in B(X)}\left(e_z-1 \right)\nonumber\\
&=&-2g+2. \nonumber
\end{eqnarray}
Here, $e_z$ is of the ramification index of $\varphi$ at $z$ and the last equation is the Riemann-Hurwitz formula.
Since we can take all $A= {\tiny \left[\left(
  \begin{array}{cccc}
    a & b \\
    c & d 
  \end{array}
  \right)
  \right]} \in \Gamma(X, u)$ with $c \not =0$, we obtain the claim.
\end{proof}
\begin{remark}
Proposition \ref{number2} is useful to estimate the number of holomorphic sections if we construct an example of flat surface with a simple Jenkins-Strebel direction.
To estimate the number of holomorphic sections, we only need estimating $\sharp {\rm Ker}(D)$ and finding some element $A\in \Gamma(X, u)$ with $c \not =0$.
It is easy to see that
${\lceil} x+1{\rceil} \leq 4{\lceil} \frac{x}{4}+\frac{1}{2}{\rceil}$
 for all $x \in \mathbb{R}$.
Hence, the inequality
\begin{center}
$\sharp S(X, u) \leq 4{\rm mod}(X)c_0 \lceil$ $\frac{b_0 \sharp {\rm Ker}(D)}{4{\rm mod}(X)}+\frac{1}{2} \rceil -2g+2$
\end{center}
always holds.
Moreover, we have the following two lemmas.
\end{remark}
Assume that $\theta=0$ is a simple Jenkins-Strebel direction.
We can estimate $\sharp {\rm Ker}(D)$ by the following way.
Let $\theta^\prime (\not=0)$ be a Jenkins-Strebel direction.
Then the rectangle $R$ is decomposed into finitely many parallelograms $R_1,R_2, \cdots, R_k$ by the $\theta^\prime$-direction saddle connections.
We label the parallelograms with respect to the rule that 
if $R_i$ and $R_j$ belong the same or congruent cylinders of Jenkins-Strebel direction $\theta^\prime$, then their labels are same.
Assume that the labels which we use are $a_1, a_2, \cdots, a_l$.
Let $a$ be the word in the free group $\left< a_1, a_2, \cdots, a_l \right>$ which is obtained by reading the labels from left to right on $R$.
We set 
\begin{center}
$n= \max \{n \in \mathbb{N}: a=b^n$ for some $b \in \left< a_1, a_2, \cdots, a_l \right>\}$.
\end{center}

\begin{lemma}\label{sgn}
$\sharp sgn^{-1}(\{1\})\leq n$.
\end{lemma}
\begin{proof}
Let $h \in sgn^{-1}(\{1\})$ such that $sgn^{-1}(\{1\})=\left< h \right>$.
Since $h$ is a translation, if $h$ sends $R_1$ to some $R_m$, then the labels of $R_i$ and $R_{i+m}$ are always same. 
This implies that $a=b^{\frac{k}{m}}$ for some $b \in \left< a_1, a_2, \cdots, a_l \right>$.
Hence, $\sharp sgn^{-1}(\{1\})=\frac{k}{m}\leq n$.
\end{proof}
\begin{lemma}\label{separate}
If a horizontal closed geodesic is separating curve of $X$,
then
\begin{center} 
${\lceil} \frac{b_0 \sharp {\rm Ker}(D)}{4{\rm mod}(X)}+\frac{1}{2}{\rceil}\leq 1$.
\end{center}
\end{lemma}
\begin{proof}
By the assumption, identifications of two horizontal sides of $R$ are independent.
If $sgn({\rm Ker}(D))=\{1\}$ then, since there exists the translation 
$z \mapsto z+ {\tiny
\left(
  \begin{array}{cccc}
    \frac{W}{\sharp{\rm Ker}(D)} \\
    0 
  \end{array}
  \right)}
$ in ${\rm Ker}(D)$, it is easy to see that 
${\tiny \left[\left(
  \begin{array}{cccc}
    1 & \frac{{\rm mod}(X)}{\sharp{\rm Ker}(D)} \\
    0 & 1 
  \end{array}
  \right)
  \right]} \in \Gamma(X, u)$.
By the definition of $b_0$, there exists $m \in \mathbb{N}$ such that $\frac{{\rm mod}(X)}{\sharp{\rm Ker}(D)}=mb_0$.
Hence, $\frac{b_0 \sharp {\rm Ker}(D)}{4{\rm mod}(X)}+\frac{1}{2} \leq \frac{1}{4m}+\frac{1}{2}$.
By the same argument, if $sgn({\rm Ker}(D))=\{\pm 1\}$, then 
${\tiny \left[\left(
  \begin{array}{cccc}
    1 & \frac{{2\rm mod}(X)}{\sharp{\rm Ker}(D)} \\
    0 & 1 
  \end{array}
  \right)
  \right]} \in \Gamma(X, u)$ and hence, $\frac{b_0 \sharp {\rm Ker}(D)}{4{\rm mod}(X)}+\frac{1}{2} \leq \frac{1}{2m}+\frac{1}{2}$ for some $m\in \mathbb{N}$.
\end{proof}

Now we come back to the proof of theorem \ref{numberthm}.
To prove this, we estimate 
$\frac{b_0}{{\rm mod}(X)}$, $\frac{{\rm mod}(X)}{b_0}$, $b_0c_0$ and $\sharp {\rm Ker}(D)$ by $g, n, p$, and $k$.
Since 
${\tiny \left[\left(
  \begin{array}{cccc}
    1 & {\rm mod}(X) \\
    0 & 1 
  \end{array}
  \right)
  \right]} \in \Gamma(X, u)$, we have $\frac{b_0}{{\rm mod}(X)} \leq 1$ by the definition of $b_0$.
We set $L_0$ and $L_1$ to be the unions of all horizontal saddle connections corresponding to lower and upper  horizontal sides of the rectangle $R$, respectively.
Let $n_0$ be the number of horizontal saddle connections contained in $L_0$.
And let $n_1$ be the number of horizontal saddle connections contained in $L_1$.
We may assume that $n_0 \geq n_1$.
 \begin{lemma}\label{average}
 The inequality
\begin{center}
$\frac{n_0+n_1}{2} \leq 2(3g-3+n)$.
\end{center}
holds.
The equality holds if and only if all punctures of $X$ are poles of $q$ of order $1$ and all zeros of $q$ is of order $1$. 
\end{lemma}
\begin{proof}
By considering the Euler characteristic of $X$, we have
\begin{center}
$\frac{n_0+n_1}{2}=2g-2+\sharp C(X, u)$.
\end{center}
Clearly, $\sharp C(X, u) \leq 4g-4+2n$.
And, $\sharp C(X, u) = 4g-4+2n$ if and only if all punctures of $X$ are poles of $q$ of order $1$ and all zeros of $q$ is of order $1$.
Hence, we obtain the claim.
\end{proof}
\begin{lemma}\label{mod/b_0}
\begin{eqnarray}
\frac{{\rm mod}(X)}{b_0} \leq \left\{
\begin{array}{ll}
8(3g-3+n)^2 & \left( sgn({\rm Ker}(D))=\{1\} \right),
\\[1em]
4(3g-3+n)^2 & \left( sgn({\rm Ker}(D))=\{\pm 1\} \right).
\end{array}
\right.
\nonumber 
\end{eqnarray}
\end{lemma}
\begin{proof}
Take $h_{A_0} \in {\rm Aff}^+(X, u)$ such that $D(h_{A_0})=A_0$.
If $sgn({\rm Ker}(D))=\{1\}$, then there exists $m_0\leq n_0+n_1$ such that $h_{A_0}^{m_0}$ fixes $L_0$ pointwise.
And there exists $m_1\leq n_1$ such that $h_{A_0}^{m_0m_1}$ fixes $L_1$ pointwise.
Then $m_0m_1b_0= k\cdot  {\rm mod}(X)$ for some $k \in \mathbb{N}$. 
Hence, by lemma \ref{average}, 
\begin{center}
$\frac{{\rm mod}(X)}{b_0}=\frac{m_0m_1}{k} \leq \frac{(n_0+n_1)^2}{2} \leq 8(3g-3+n)^2$.
\end{center}
If $sgn({\rm Ker}(D))=\{ \pm 1\}$, we may assume that $h_{A_0}(L_i)=L_i$ $(i=0,1)$.
There exists $m_0\leq n_0$ such that $h_{A_0}^{m_0}$ fixes $L_0$ pointwise.
And there exists $m_1\leq n_1$ such that $h_{A_0}^{m_0m_1}$ fixes $L_1$ pointwise.
Then $m_0m_1b_0= k\cdot  {\rm mod}(X)$ for some $k \in \mathbb{N}$. 
Hence, 
\begin{center}
$\frac{{\rm mod}(X)}{b_0}=\frac{m_0m_1}{k} \leq n_0n_1 \leq \frac{(n_0+n_1)^2}{4} \leq 4(3g-3+n)^2$.
\end{center} 
\end{proof}
 \begin{lemma}\label{ker}
 \begin{eqnarray}
\sharp {\rm Ker}(D) \leq \left\{
\begin{array}{ll}
2(3g-3+n) & \left( sgn({\rm Ker}(D))=\{1\} \right),
\\[1em]
4(3g-3+n) & \left( sgn({\rm Ker}(D))=\{\pm 1\} \right).
\end{array}
\right.
\nonumber 
\end{eqnarray}
\end{lemma}
\begin{proof}
Let $h \in sgn^{-1}(\{1\})$.
Then $h(L_i)=L_i$ $(i=0,1)$ and there exists  $m \leq n_1$ such that $h^m$ fixes $L_1$ pointwise.
Since $h$ is a translation, $h^m=id$.
Thus, 
\begin{center}
$\sharp sgn^{-1}(\{1\}) \leq m \leq \frac{n_0+n_1}{2} \leq 2(3g-3+n)$
\end{center}
by lemma \ref{average}.
\end{proof}
\begin{remark}
For the flat surface $X_3$ as in example \ref{D-example2}, the equations of lemma \ref{average} and lemma \ref{ker} hold. 
\end{remark}
 Finally, to estimate $b_0c_0$, we prove a new property of finitely generated Fuchsian groups.
\begin{theorem} \label{fuchsian}
Let $\Gamma$ be a Fuchsian group of type $(p,k: \nu_1, \cdots, \nu_k)$.
Here,  $\nu_i \in \{2, \cdots, \infty\}$.
Assume that $k_0$ is the number of $\nu_i$'s which are equal to $\infty$.
If $\Gamma$ contains $ \tiny \left[\left(
  \begin{array}{cccc}
    1 & 1 \\
    0 & 1
  \end{array}
  \right)
  \right]$ and it is not a power of other elements of $\Gamma$, then there exists  
  $ \tiny \left[\left(
  \begin{array}{cccc}
    a & b \\
    c & d
  \end{array}
  \right)
  \right] \in \Gamma$
  such that 
\begin{equation}
\displaystyle
1\leq |c| <{\rm Area}(\mathbb{H}/\Gamma)-k_0+1.
\nonumber
\end{equation}
Here, ${\rm Area}(\mathbb{H}/\Gamma)$ is the hyperbolic area of $\mathbb{H}/\Gamma$.
\end{theorem}
\begin{remark}
It is known that 
\begin{center}
$\displaystyle {\rm Area}(\mathbb{H}/\Gamma)=2\pi\left(2p-2+\sum_{i=1}^k(1-\frac{1}{\nu_i})\right)$
\end{center}
if $\Gamma$ is a Fuchsian group of type $(p, k: \nu_1, \cdots, \nu_k)$.
See \cite{FarKra92}.
\end{remark}
\begin{proof}
Set $A={\rm Area}(\mathbb{H}/\Gamma)-k_0+1$.
Assume that $|c|\geq A$ or $c=0$ for all 
$\tiny \left[\left(
  \begin{array}{cccc}
    a & b \\
    c & d
  \end{array}
  \right)
  \right] \in \Gamma$.
Set $\Gamma_0=\{{\tiny \left[\left(
  \begin{array}{cccc}
    a & b \\
    c & d
  \end{array}
  \right)
  \right]} \in \Gamma: c \not =0\}$. 
We consider the Ford region $D$ of $\Gamma$. 
Here, the Ford region $D$ is defined by
\begin{center}
$\displaystyle D=\{ z \in\mathbb{H}: |{\rm Re}(z)|<\frac{1}{2}\} \cap \bigcap_{B \in \Gamma_0}D_B$
\end{center}
and
\begin{center}
$D_B=\{ z \in\mathbb{H}: |z+\frac{d}{c}|>\frac{1}{|c|}\}$.
\end{center}
By the assumption, $D$ contains the region $D^\prime=\{z\in\mathbb{H}: |{\rm Re}(z)|<\frac{1}{2}, {\rm Im}(z)>\frac{1}{A}\}$.
Let $v_1, \cdots, v_{k_0}$ be punctures of $\mathbb{H}/\Gamma$.
Assume that $\infty \in \overline{\mathbb{H}}$ corresponds to the puncture $v_{k_0}$.
There exists the horodisk $U_{k_0}$  around $v_{k_0}$ which is $D^\prime$ in $D$.
And, by lemma \ref{Shimizu}, there exist horodisks $U_i$ around $v_i$ which are disjoint each other and their hyperbolic areas are $1$ for all $i=1, \cdots, k_0-1$.
Since every domain as figure \ref{Ford2} has area $\pi-2$, we conclude that $U_i \cap U_{k_0}=\phi$ for all $i=1, \cdots, k_0-1$.
Hence, 
\begin{eqnarray}
{\rm Area}(\mathbb{H}/\Gamma) &>& \int_{D^\prime}\frac{dxdy}{y^2}+(k_0-1)\cdot 1 \nonumber\\
&=& A+k_0-1 \nonumber \\
&=& {\rm Area}(\mathbb{H}/\Gamma). \nonumber
\end{eqnarray}
This is a contradiction.
\end{proof}
\begin{figure}[h]
 \begin{center}
  \includegraphics[keepaspectratio, scale=1]{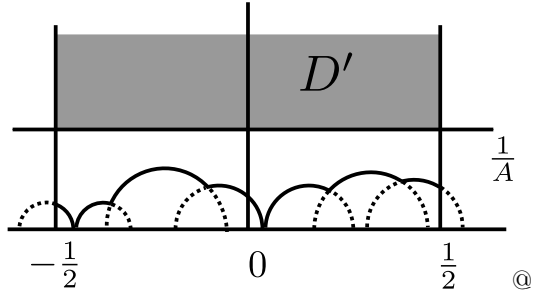}@
  \caption{The region $D^\prime$ contained in the Ford region $D$ of $\Gamma$}
\label{Ford}
 \end{center}
\end{figure}
\begin{figure}[h]
 \begin{center}
  \includegraphics[keepaspectratio, scale=1]{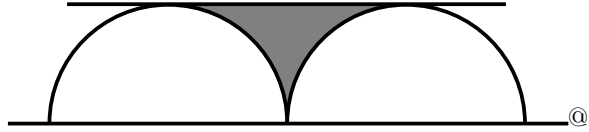}@
  \caption{A region in $\mathbb{H}$ corresponding to a neighborhood of a puncture of $\mathbb{H}/\Gamma$ of area $\pi-2$ }
\label{Ford2}
 \end{center}
\end{figure}

\begin{remark}
Let $\Gamma$ be a Fuchsian group of type $(p,k: \infty, \cdots , \infty )$ with $k \geq 1$. 
We suppose that $\Gamma$ satisfies the assumptions of theorem \ref{fuchsian}.
Take the commutator $Q$
of  
${\tiny \left[\left(
  \begin{array}{cccc}
    1 & 1 \\
    0 & 1
  \end{array}
  \right)
  \right]}, {\tiny \left[\left(
  \begin{array}{cccc}
    a & b \\
    c & d
  \end{array}
  \right)
  \right]} \in \Gamma$.
Since ${\rm tr}(Q)=c^2+2>2$, theorem \ref{fuchsian} implies that the Riemann surface $\mathbb{H}/\Gamma$ has a closed geodesic whose length is less than $2 \cosh^{-1}\left( \frac{\left(2 \pi (2p-2+k)-k+1 \right)^2}{2} +1\right)$. 
If this geodesic is simple, it has a collar neighborhood whose width is greater than 
$2 \sinh^{-1}\left(\frac{2}{\gamma(p, k)\sqrt{\gamma(p, k)^2+4}} \right)$, where $\gamma(p, k)=2 \pi (2p-2+k)-k+1$.
\end{remark}

\begin{proof}[Proof of theorem \ref{numberthm}]
If $sgn({\rm Ker}(D))=\{1\}$, by proposition \ref{number2}, lemma \ref {mod/b_0}, lemma \ref {ker}, and theorem \ref{fuchsian},
\begin{eqnarray}
\sharp S(X, u) &\leq&  8(3g-3+n)^2{\rm Area}(\mathbb{H}/\Gamma(X, u))(6g-5+2n)-2g+2 \nonumber\\
&\leq & 32\pi(2p-2+k)(3g-3+n)^2(3g-\frac{5}{2}+n)-2g+2. \nonumber
\end{eqnarray}
And, if  $sgn({\rm Ker}(D))=\{\pm 1\}$,
\begin{eqnarray}
\sharp S(X, u) &\leq&  16(3g-3+n)^2{\rm Area}(\mathbb{H}/\Gamma(X, u))(3g-2+n)-2g+2 \nonumber\\
&\leq & 32\pi(2p-2+k)(3g-3+n)^2(3g-2+n)-2g+2. \nonumber
\end{eqnarray}
\end{proof}
\section*{Acknowledgments}
This work was supported by Global COE Program ``Computationism as a Foundation for the Sciences".
The author is grateful to Hiroshige Shiga for his valuable comments and suggestions. 
The author is also grateful to Curtis T McMullen for telling me related papers. 
Thanks to Yunping Jiang and Sudeb Mitra for their suggestions about theorem \ref{fuchsian}.
\nocite{GutHubSch03}
\nocite{Moller06}
\nocite{Veech91}
\bibliography{ref}

\begin{thebibliography}{GHS03}

\bibitem[Ber73]{Bers73}
L.~Bers.
\newblock {Fiber spaces over Teichm{\"u}ller spaces}.
\newblock {\em Acta Math.}, 130:89--126, 1973.

\bibitem[EG97]{EarGar97}
C.~Earle and F.~Gardiner.
\newblock {Teichm{\"u}ller disks and Veech's F-structures}.
\newblock {\em Contemp. Math.}, 201:165--189, 1997.

\bibitem[FK92]{FarKra92}
H.~M. Farkas and I.~Kra.
\newblock {\em {Riemann Surfaces}}, volume~71 of {\em Graduate Texts in
  Mathematics}.
\newblock Springer-Verlag, New York, 1992.

\bibitem[For25]{Ford25}
L.~R. Ford.
\newblock {The fundamental region for a Fuchsian group}.
\newblock {\em Bull. Amer. Math. Soc.}, 31(9--10):531--539, 1925.

\bibitem[GHS03]{GutHubSch03}
E.~Gutkin, P.~Hubert, and T.~Schmidt.
\newblock {Affine diffeomorphisms of translation surfaces: periodic points,
  Fuchsian groups, and arithmeticity}.
\newblock {\em Ann. Sci. {\'E}cole Norm. Sup.}, 36(6):847--866, 2003.

\bibitem[HS07]{HerSch07}
F.~Herrlich and G.~Schmith{\"u}sen.
\newblock {\em {On the boundary of Teichm{\"u}ller disks in Teichm{\"u}ller and
  in Schottky space}}, volume~11 of {\em IRMA Lect. Math. Theor. Phys.},
  chapter~6, pages 293--349.
\newblock Eur. Math. Soc., Zurich, 2007.

\bibitem[IS88]{ImaShi88}
Y.~Imayoshi and H.~Shiga.
\newblock {A finiteness theorem for holomorphic families of Riemann surfaces}.
\newblock In {\em Holomorphic functions and moduli, Vol. II}, number~11 in
  Math. Sci. Res. Inst. Publ., pages 207--219. Springer-Verlag, New York, 1988.

\bibitem[IT92]{ImaTan92}
Y.~Imayoshi and M.~Taniguchi.
\newblock {\em {An Introduction to Teichm{\"u}ller Spaces}}.
\newblock Springer-Verlag, Tokyo, 1992.

\bibitem[KH96]{KatHas96}
A.~Katok and B.~Hasselblatt.
\newblock {\em Introduction to the Modern Theory of Dynamical Systems}.
\newblock Cambridge University Press, Cambridge, 1996.

\bibitem[Kob05]{Kobayashi05}
S.~Kobayashi.
\newblock {\em {Hyperbolic manifolds and holomorphic mappings}}.
\newblock World Scientific Publishing Co. Pte. Ltd., Hackensack, NJ, 2005.

\bibitem[M{\"o}l06]{Moller06}
M.~M{\"o}ller.
\newblock {Periodic points on Veech surfaces and the Mordell-Weil group over a
  Teichm{\"u}ller curve}.
\newblock {\em Invent. Math.}, 165(3):633--649, 2006.

\bibitem[Roy71]{Royden71}
H.~L. Royden.
\newblock {\em {Automorphisms and isometries of Teichm{\"u}ller space}}, pages
  369--383.
\newblock Number~66 in Ann. of Math. Studies. Princeton Univ. Press, Princeton,
  N.J, 1971.

\bibitem[Shi97]{Shiga97}
H.~Shiga.
\newblock {On monodromies of holomorphic families of Riemann surfaces and
  modular transformations}.
\newblock {\em Math. Proc. Cambridge Philos. Soc.}, 122:541--549, 1997.

\bibitem[Shi11]{Shinomiya12}
Y.~Shinomiya.
\newblock {Veech groups of flat structures on Riemann surfaces}.
\newblock {\em ArXiv}, 1105.5790:1--20, 2011.

\bibitem[Sin72]{Singerman72}
D.~Singerman.
\newblock {Finitely maximal Fuchsian groups}.
\newblock {\em J. London Math. Soc.}, 6:29--38, 1972.

\bibitem[Vee89]{Veech89}
W.~Veech.
\newblock {Teichm{\"u}ller curves in moduli space, Eisenstein series and an
  application to triangular billiards}.
\newblock {\em Invent. Math.}, 97(3):553--583, 1989.

\bibitem[Vee91]{Veech91}
W.~Veech.
\newblock {Erratum: Teichm{\"u}ller curves in moduli space, Eisenstein series
  and an application to triangular billiards}.
\newblock {\em Invent. Math.}, 103(2):447, 1991.

\end{thebibliography}

\end{document}